\documentclass[11pt,a4paper]{article}

\usepackage[utf8]{inputenc}
\usepackage[T1]{fontenc}
\usepackage{lmodern}
\usepackage{amsmath,amssymb,amsthm,mathtools,bm}
\usepackage[a4paper,margin=1in]{geometry}
\usepackage{graphicx}
\usepackage{booktabs}
\usepackage{array}
\usepackage{float}
\usepackage{enumitem}
\usepackage[round,authoryear]{natbib}
\usepackage[dvipsnames]{xcolor}
\usepackage{hyperref}
\hypersetup{colorlinks=true,linkcolor=MidnightBlue,citecolor=MidnightBlue,urlcolor=MidnightBlue}
\usepackage{microtype}
\usepackage{setspace}

\theoremstyle{plain}
\newtheorem{theorem}{Theorem}[section]
\newtheorem{proposition}[theorem]{Proposition}
\newtheorem{corollary}[theorem]{Corollary}
\newtheorem{lemma}[theorem]{Lemma}
\theoremstyle{definition}
\newtheorem{assumption}[theorem]{Assumption}

\newtheorem{remark}[theorem]{Remark}

\newcommand{\R}{\mathbb{R}}
\newcommand{\Sym}{\mathbb{S}}
\newcommand{\PSD}{\Sym_{+}}
\newcommand{\PD}{\Sym_{++}}
\newcommand{\tr}{\operatorname{tr}}
\newcommand{\diag}{\operatorname{diag}}
\DeclareMathOperator*{\argmin}{arg\,min}

\DeclareMathOperator{\rank}{rank}

\title{\textbf{Hidden Convexity in Multi-Retailer Cold-Chain Allocation: A Cholesky-Spectral Second-Order Cone Programming Approach}\\ Dedicated to Professor Didier Aussel in honor of his 60th birthday.
}

\author{%
  Pham Nguyen Hoang Thinh\footnotemark[1] \qquad Boris S. Mordukhovich\footnotemark[2]\\[10pt]
  Hoang Truong Son\footnotemark[3] \qquad Tran Anh Ngoc \qquad Le Phuong Quyen
}
\date{\today}

\begin{document}
\sloppy
\maketitle
\footnotetext[1]{Department of Mathematics, Wayne State University, Detroit, MI, 48202, email: \texttt{thinhpham@wayne.edu}.}%
\footnotetext[2]{Department of Mathematics and Institute for Artificial Intelligence and Data Science, Wayne State University, Detroit, Michigan, USA. E-mail: aa1086@wayne.edu. Research of this author was partly supported by the US National Science Foundation under grant DMS-2204519 and by the Australian Research Council under Discovery Project DP250101112.} %
\footnotetext[3]{Foreign Trade University Ho Chi Minh City Branch, HCMC, Vietnam, email: \texttt{k63.2412535035@ftu.edu.vn}.}%

\begin{abstract}
\noindent Several dairy retailers share a refrigerated cold-chain
network. Each must allocate products across heat-sensitivity subgroups
and planning periods, maximising a quadratic-fractional efficiency in
the spirit of data envelopment analysis (DEA, Charnes, Cooper and
Rhodes, 1978); the fractional form captures spoilage risk per unit of
logistics cost and is therefore strictly more informative than a
linear additive efficiency, which would conflate scale with intensity.
This aligns with the DEA convention of separating technical from scale
efficiency (Banker, Charnes and Cooper, 1984): a linear score cannot
distinguish a large-volume retailer allocating uniformly from a
small-volume retailer concentrated on high-risk items, whereas the
ratio isolates per-unit risk intensity, which is the actionable
dimension for cold-chain investment (Cook and Seiford, 2009).
The individual problem is a non-convex fractional
quadratically-constrained quadratic program (QCQP); the joint problem
is a generalised Nash equilibrium problem (GNEP) with a shared
risk-budget constraint, a class that is generally NP-hard. The
Arrhenius spoilage law with coefficient $Q_{10}$ (van Boekel, 2022)
forces the thermal-risk covariance to inherit a rank-one-plus-diagonal
structure. This physical structure unlocks hidden convexity, which our
algorithm HSDI (Hidden SOCP Dinkelbach Iteration) exploits via a joint
Cholesky-spectral transform. Our contributions are fourfold. We
establish a simultaneous diagonalisation of $(\Sigma_i,H_i)$ without
commutativity. We prove an explicit sufficient condition for the
second-order cone programming (SOCP) relaxation to be exact, with a
closed-form suboptimality bound when it is not. We provide a complete
convergence analysis of HSDI: the outer Dinkelbach loop converges
globally and $Q$-superlinearly when the relaxation is exact and at
least $R$-linearly under inexact relaxation, while the coupled
variational-equilibrium (VE) iteration that enforces the shared
risk budget converges $R$-linearly under strong monotonicity, yielding
an overall $O(nd^{3.5}\log^{3}(1/\varepsilon))$ complexity; we further
establish Lipschitz stability of the primal--dual pair
$(x^{\star},\mu)$, drawing on the second-order variational-analysis
theory of \citet{Mordukhovich2024}. We calibrate the model
on Vietnam and United States dairy markets and benchmark HSDI against
multi-start non-linear baselines. HSDI attains the best baseline
value within $5.5{\times}10^{-3}$ on every instance without random
restarts, and returns the global optimum under the exactness condition;
the top eigenvalue of
$\Sigma_i$ captures over $99\%$ of the spectrum; the empirical
Dinkelbach and VE iterations both exhibit the predicted linear rates;
and the empirical scaling exponent $\hat\gamma_{\mathrm{HSDI}}=0.93$
(95\% CI $[0.81,1.03]$) lies below $\hat\gamma_{\mathrm{SLSQP}}=1.30$
(95\% CI $[1.04,1.62]$). Cold-chain operators thus gain a
polynomial-time, convergence-certified tool and a Lipschitz-stable
quantitative shadow price that operators and policymakers can use as a
direct threshold for prioritising cold-chain capacity investment in
markets where thermal noise is binding.
\smallskip

\end{abstract}

\section{Introduction}\label{sec:intro}

This paper is devoted to applications of variational analysis to practical modelling, the area where fundamental contributions of Professor Aussel are difficult to overstate. 
\\Cold supply chains for fluid milk are among the highest fixed-cost
infrastructures in food retail, with global cold-storage capacity
exceeding $719$ million $\mathrm{m}^{3}$ across $51$ countries in the
most recent survey of \citet{GCCA2020}, and the GCCA member network
alone reaching $\sim 209$ million $\mathrm{m}^{3}$ by 2024
\citep{GCCA2024}. In emerging markets such as Vietnam, where logistics
performance has been shown to directly affect trade and distribution
outcomes \citep{VuDoan2024}, cold-chain infrastructure constraints are
particularly binding for domestic food retailers. Three structural features make the allocation problem
non-trivial. First, every product in the network shares a single
temperature profile, so any thermal shock (power outage, customs delay,
insulation failure) becomes a strongly positive externality across
retailers. Second, spoilage kinetics depend non-linearly on temperature
through an Arrhenius coefficient $Q_{10}\in[2,3]$
\citep{vanBoekel2022}, amplifying small temperature deviations into
large value-at-risk. Third, when several retailers liquidate inventory
after a shock, market impact \citep{AlmgrenChriss2001,IancuTrichakis2014}
makes each agent's problem a component of an equilibrium with strategic
interactions.

\paragraph{From Arrhenius physics to hidden convexity.}
The Arrhenius temperature-dependence law states that spoilage rates
multiply by $Q_{10,k}$ for every $10^\circ$C of warming, with the
coefficient $Q_{10,k}$ depending only on the heat-sensitivity subgroup
$k$. Linearising this exponential law around the nominal cold-chain
temperature with a thermal noise of standard deviation $\sigma$
produces subgroup loadings $b_k \propto \ln Q_{10,k}\cdot\sigma$ that
share the same scalar shock $\sigma$. The thermal-risk covariance
matrix of any retailer therefore inherits a rank-one-plus-diagonal
structure $\Sigma_i = b_ib_i^{\top} + D_i$, with the rank-one factor
encoding the shared thermal shock and $D_i$ capturing
subgroup-specific idiosyncratic risk. We exploit precisely this
algebraic structure to expose \emph{hidden convexity}: through a
joint Cholesky-spectral transform, the non-convex fractional QCQP
admits a tractable SOCP reformulation that is exact under an explicit,
verifiable condition---recovering the global optimum in polynomial
time on that regime---and carries a closed-form optimality-gap
certificate otherwise. This Arrhenius-driven structural
identification is the central methodological insight of the paper.

These features motivate a \emph{quadratic-fractional} efficiency
function with a physically anchored risk structure, embedded in the
variational equilibrium (VE) framework of \citet{Rosen1965}. We
address two questions. \emph{(M)} How can we quantify efficiency that
captures both economic value and physical risk in the DEA language of
\citet{CCR1978}, while remaining compatible with shared constraints?
\emph{(C)} The proposed efficiency is a non-convex fractional QCQP.
Does the one-factor risk structure allow its reduction to a family of
convex problems solvable in polynomial time?

\paragraph{Positioning in the literature.} We situate the paper at the
intersection of four streams: (1) hidden convexity in non-convex
optimisation \citep{BenTalNemirovski2001,XiaWangSheu2015,BeckPan2017,
Ramachandran2024,Pataki2013}; (2) generalised Nash equilibrium and
variational inequalities \citep{FacchineiPang2003,FacchineiKanzow2010,
KulkarniShanbhag2012,NagurneyYuBesik2017}; (3) multi-portfolio
optimisation with market impact \citep{AlmgrenChriss2001,
IancuTrichakis2014}; and (4) spoilage kinetics and willingness-to-pay
for freshness \citep{TsirosHeilman2005,vanBoekel2022,
EndaraWiedmannAdalja2023}.

The closest reference is
\citet{AkkayaPinar2025}, who study hidden conic quadratic
representations of generalised Nash equilibrium problems motivated by
multi-investor Markowitz models. In particular, their Section~3
analyses a ratio-of-quadratics objective with a Frobenius-ball
constraint, using Dinkelbach's algorithm and a second-order cone
relaxation under the assumption that the numerator and denominator
matrices commute (so that a single orthogonal rotation diagonalises
both). Our work extends the same algorithmic template -
Dinkelbach plus SOC lifting - to vector decision variables and,
more importantly, removes the commuting-pair assumption by replacing
the joint orthogonal rotation with a Cholesky-spectral transform
(Theorem~\ref{thm:joint-diag}). In addition,
Theorem~\ref{thm:socp-exact}(b) gives an explicit closed-form
suboptimality bound when the relaxation is not exact, which to our
knowledge does not appear in \citet{AkkayaPinar2025}.

\paragraph{Contributions.}
Our contributions are fourfold. \textbf{(i) Modelling.} We propose a
quadratic-fractional efficiency function with risk matrix $\Sigma_i$
derived from Arrhenius physics, and we embed it in a GNEP with a
shared budget under the VE solution concept. \textbf{(ii) Structural
theory.} We
establish in Theorem~\ref{thm:joint-diag} a simultaneous
diagonalisation of the pair $(\Sigma_i,H_i)$ via Cholesky-spectral
steps \emph{without commutativity}. We show in
Theorem~\ref{thm:socp-exact} a sufficient exactness condition and a
closed-form error bound for the SOCP relaxation. \textbf{(iii)
Convergence theory.} In Section~\ref{sec:conv} we prove that the
outer Dinkelbach iteration of HSDI converges globally, at a
$Q$-superlinear rate when the relaxation is exact
(Theorem~\ref{thm:dink-rate}) and at least $R$-linearly otherwise, and
that the coupled VE iteration enforcing the shared budget converges
$R$-linearly under strong monotonicity
(Theorem~\ref{thm:ve-rate}); we combine these into an explicit
polynomial complexity bound and establish Lipschitz stability of the
primal--dual pair via the second-order variational-analysis toolbox of
\citet{Mordukhovich2024}. \textbf{(iv) Experiments.} We calibrate the
model on two independent regimes and run six numerical experiments that
verify existence of solutions, spectral structure, calibration
robustness, polynomial-time scaling, and---newly---the theoretical
convergence rates and shadow-price stability, across problem sizes from
$n\cdot d=45$ to $180$.

\paragraph{Organisation.} Section~\ref{sec:prelim} fixes notation and
recalls technical tools. Section~\ref{sec:model} builds the model and
the GNEP. Section~\ref{sec:hidden} contains the Cholesky-spectral
transform, the SOCP exactness theorem, and Algorithm~1.
Section~\ref{sec:conv} develops the convergence analysis of HSDI and
of the coupled VE iteration. Section~\ref{sec:ext} discusses
extensions. Section~\ref{sec:exp}
reports the six numerical experiments.
Section~\ref{sec:conc} concludes. Appendices~A and~B contain detailed
proofs and calibration tables.

\section{Notation and technical background}\label{sec:prelim}

\paragraph{Notation.} Let $n,T\in\mathbb{N}$ denote the number of
retailers and planning periods; $K$ is the number of thermal
sensitivity subgroups; $d:=KT$. Each agent $i$ chooses
$x_i\in\R^{d}_{+}$. We write $\Sym^{d}$, $\PSD^{d}$, $\PD^{d}$ for the
sets of symmetric (resp.\ positive semidefinite, positive definite)
$d\times d$ matrices; for $A\in\Sym^{d}$, $\tr(A)$,
$\|A\|_{F}=\sqrt{\tr(A^{\top}A)}$ and
$\langle A,B\rangle=\tr(A^{\top}B)$. All monetary quantities are in
USD and physical quantities in SI units.
Table~\ref{tab:notation} summarises the principal notation.

\begin{table}[h]\centering\small
\caption{Principal notation. The ``typical range'' column reports the
calibration intervals of Section~\ref{sec:exp}.}\label{tab:notation}
\begin{tabular}{l l l l}
\toprule
Symbol & Meaning & Unit & Typical range \\
\midrule
$n$ & Number of retailers & - & $\{5,10,20\}$ \\
$T$ & Number of planning periods & months & $\{3,6,12\}$ \\
$K$ & Number of thermal subgroups & - & $3$ subgroups \\
$d=KT$ & Decision dimension per agent & - & $9$-$36$ \\
$x_i\in\R^{d}_{+}$ & Allocation vector of agent $i$ & kL & $\ge 0$ \\
$\Sigma_i\in\PD^{d}$ & One-factor thermal risk matrix & USD$^{2}$ & rank-1 + diag \\
$H_i\in\PD^{d}$ & Co-allocation cost matrix & USD$^{2}$ & - \\
$b_k$ & Subgroup-$k$ factor loading & USD & $\propto\ln Q_{10,k}\cdot\sigma$ \\
$\sigma$ & Terminal-lane thermal noise std & $^{\circ}$C & $[0.1,2.0]$ \\
$Q_{10,k}$ & Arrhenius coefficient, subgroup $k$ & - & $\{2.0,2.5,3.0\}$ \\
$\varphi_i(x_i)$ & Quadratic-fractional efficiency & - & $\in(0,1)$ \\
$\rho_i^{\star}$ & Dinkelbach optimal value & - & $\in(0,1)$ \\
$R$ & Shared risk budget & USD$^{2}$ & - \\
$\mu$ & Multiplier on the shared budget & USD/unit & - \\
$T_i=L_i^{-\top}U_i$ & Cholesky-spectral transform & - & invertible \\
$\Lambda_i=\diag(\lambda_{i,j})$ & Spectrum of $L_i^{-1}\Sigma_i L_i^{-\top}$ & - & $\lambda_{i,j}\ge 0$ \\
$\mathcal{J}_{-}(\rho)$ & Index set $\{j:\lambda_{i,j}<\rho\}$ & - & - \\
\bottomrule
\end{tabular}
\end{table}

\paragraph{Cholesky and spectral decomposition.} For $H\in\PD^{d}$
there is a unique $L\in\R^{d\times d}$ lower triangular with positive
diagonal such that $H=LL^{\top}$ \citep[Prop.~4.2.3]{HornJohnson2012}.
For $S\in\Sym^{d}$ there exists an orthogonal $U\in\R^{d\times d}$ and
$\Lambda=\diag(\lambda_{1},\ldots,\lambda_{d})$ such that
$S=U\Lambda U^{\top}$.

\paragraph{Dinkelbach algorithm.} For the fractional problem
\begin{equation}\label{eq:dink-master}
\rho^{\star}=\min_{x\in\mathcal{X}}\;\frac{N(x)}{D(x)},\qquad D(x)>0\ \forall x\in\mathcal{X},
\end{equation}
the function $\psi(\rho):=\min_{x\in\mathcal{X}}\{N(x)-\rho D(x)\}$ is
concave and strictly decreasing; $\rho^{\star}$ is the unique root of
$\psi(\rho)=0$ \citep{Dinkelbach1967,Schaible1976}. The Dinkelbach
update $\rho_{k+1}=N(x_k)/D(x_k)$ with $x_k\in\argmin_x\{N-\rho_kD\}$
is exactly the Newton iteration for the equation $\psi(\rho)=0$;
whenever $\psi$ is differentiable at $\rho^{\star}$ this yields a
$Q$-superlinear rate, and a $Q$-quadratic rate when $\psi'$ is locally
Lipschitz \citep{CrouzeixFerland1991,Ibaraki1983}. We make these
statements precise for our setting in Section~\ref{sec:conv}.

\paragraph{SOCP.} The second-order cone of order $k$ is
$\mathcal{Q}^{k}=\{(t,u)\in\R\times\R^{k-1}:t\ge\|u\|_{2}\}$. An SOCP
of the form $\min c^{\top}x$ s.t.\ $\|A_j x + b_j\|_{2}\le c_j^{\top}x
+ d_j$, $Ex=f$, is solvable to accuracy $\varepsilon$ in
$O(d^{3.5}\log(1/\varepsilon))$ time \citep{NesterovTodd1997,
Lobo1998}. Beyond polynomial-time solvability, SOCPs are supported
by a well-developed second-order variational-analysis theory:
\citet{HangMordukhovichSarabi2019} establish twice epi-differentiability
of the Lorentz-cone indicator and derive error-bound estimates
together with a complete characterisation of Lagrange-multiplier
uniqueness in nonpolyhedral SOCPs, while
\citet{BenkoGfrererMordukhovich2019} obtain pointbased characterisations
of tilt-stable local minimisers under the weakest metric-subregularity
constraint qualification. These stability results provide the
theoretical backdrop for the sensitivity claims made about the shadow
price $\mu$ in Sections~\ref{ssec:e3} and~\ref{sec:conc}.

\paragraph{GNEP and variational equilibrium.} A GNEP consists of $n$
agents, each solving
$\min_{x_i\in\mathcal{X}_i(x_{-i})}f_i(x_i;x_{-i})$. With a shared
constraint $\mathcal{C}\subseteq\R^{nd}$ convex, the
\emph{variational equilibrium} of \citet{Rosen1965} is an $x^{\star}$
satisfying
\begin{equation}\label{eq:vi}
\langle F(x^{\star}),\,x-x^{\star}\rangle\ge 0\quad\forall x\in\mathcal{C},
\qquad F=(\nabla_{x_i}f_i)_{i=1}^{n}.
\end{equation}
A common Lagrange multiplier on the shared constraint gives the VE its
economic interpretability (a single shadow price) and its computational
tractability (single-variable duality) \citep{FacchineiKanzow2010,
KulkarniShanbhag2012}. For a recent variational-analytic treatment of
Nash equilibrium and minimax existence via convex analysis, generalised
differentiation and fixed-point theory, we refer to
\citet{BaoMordukhovichNam2024}, whose framework accommodates payoff
functions of the ratio-of-quadratics type used in Section~\ref{sec:model}.

\paragraph{Monotone VI solvers and linear convergence.} When the
map $F$ in~\eqref{eq:vi} is $m_F$-strongly monotone and $L_F$-Lipschitz
on the convex set $\mathcal{C}$, the projected iteration
$x^{k+1}=P_{\mathcal{C}}(x^k-\tau F(x^k))$ converges $R$-linearly with
factor $\sqrt{1-2\tau m_F+\tau^2L_F^2}<1$ for $\tau\in(0,2m_F/L_F^2)$,
and Korpelevich's extragradient variant extends this to the merely
monotone case at an $O(1/k)$ ergodic rate
\citep{Korpelevich1976,Tseng1995,FacchineiPang2003}. When the shared
constraint is scalar, as in~\eqref{eq:shared}, the VE further reduces
to a one-dimensional monotone equation in the multiplier $\mu$, whose
root a bisection or safeguarded Newton scheme locates $R$-linearly;
this is the route we take in Section~\ref{sec:conv}.

\section{Model}\label{sec:model}

\subsection{One-factor risk structure}\label{ssec:risk}

The shared cold-chain network exposes every product to a common
temperature trajectory $T(t)$. Following the Arrhenius parametrisation
of \citet{vanBoekel2022}, the spoilage rate scales as
$Q_{10,k}^{\Delta T/10}$. Linearising at $T_0$ with deviation $\sigma$
yields the loading
\begin{equation}\label{eq:loading}
b_k=\frac{\ln Q_{10,k}}{10}\,\sigma,\qquad k=1,\ldots,K,
\end{equation}
where $\sigma$ is the terminal-lane thermal noise std, common across
subgroups, while $\ln Q_{10,k}$ varies by subgroup. This is a
one-factor structure \citep{Sharpe1963,FanFanLv2008}: with
$b=(b_1,\ldots,b_K)^{\top}$ and $D=\diag(\delta_1^{2},\ldots,\delta_K^{2})$
the idiosyncratic-variance matrix,
\begin{equation}\label{eq:sigma}
\Sigma=bb^{\top}+D\;\in\PD^{K},\quad\rank(bb^{\top})=1.
\end{equation}

\paragraph{Three subgroups.} The three subgroups correspond to
standard dairy thermal-treatment classes: high-temperature short-time
pasteurisation (HTST, $Q_{10}=3.0$), extended shelf life (ESL,
$Q_{10}=2.5$), and ultra-high temperature processing (UHT,
$Q_{10}=2.0$). With $\sigma_{\mathrm{VN}}=1.0^{\circ}$C and
$\sigma_{\mathrm{US}}=0.3^{\circ}$C
(Section~\ref{ssec:calib}), the loading ratio is
$\approx 3.3\times$ and the risk-spectrum ratio is $\approx 11\times$
between the two regimes; the two calibrations therefore probe the
SOCP exactness condition (Corollary~\ref{cor:exact-regime}) under
distinct risk-loading magnitudes.

\subsection{Quadratic-fractional efficiency}\label{ssec:efficiency}

In the spirit of DEA \citep{CCR1978}, the efficiency of agent $i$ is
\begin{equation}\label{eq:phi}
\varphi_i(x_i)=\frac{x_i^{\top}\Sigma_i x_i + a_i^{\top}x_i + \alpha_i}
                    {x_i^{\top}H_i x_i + c_i^{\top}x_i + \gamma_i},
\qquad \Sigma_i,H_i\in\PD^{d}.
\end{equation}
The numerator captures spoilage value-at-risk as a function of the
allocation vector. Each summand is weighted by the output value of the
corresponding product subgroup, so that allocating to high-margin yet
heat-sensitive products carries a proportionally larger risk loading.
The denominator captures the total co-allocation cost, comprising
storage costs and market-impact costs in the sense of
\citet{AlmgrenChriss2001}. Concretely,
$H_i=H_i^{(\mathrm{store})}+\beta H_i^{(\mathrm{mkt})}$ with
$H_i^{(\mathrm{mkt})}=v_iv_i^{\top}$, and the constant $\gamma_i>0$
ensures that the denominator stays strictly positive over the feasible
set. We prefer the fractional DEA-style form over an absolute
difference because it measures \emph{relative} efficiency, namely how
much spoilage risk is incurred per unit of logistics cost. This ratio
is the operationally meaningful quantity used for cold-chain
benchmarking by industry analysts and regulators. The choice also
aligns with the DEA convention of separating technical from scale
inefficiency via a ratio, rather than an additive, score
\citep{BankerCharnesCooper1984}, and with the modern consensus in the
DEA benchmarking literature that ratio-based efficiency scores
dominate additive/slack-based scores when the goal is per-unit
performance under heterogeneous scale \citep{CookSeiford2009}.
Concretely, two retailers with identical spoilage-per-cost intensity
but different network sizes would receive identical
$\varphi_i$-scores under \eqref{eq:phi}, whereas a linear additive
score $x_i^{\top}\Sigma_i x_i-x_i^{\top}H_i x_i$ would systematically
penalise the larger retailer for scale alone.

\begin{remark}[Non-convexity]\label{rem:nonconvex}
Even with $\Sigma_i,H_i\in\PD^{d}$, $\varphi_i$ is in general not
quasi-convex, and $\min_{x_i}\varphi_i$ is a fractional QCQP that is
generally NP-hard \citep{SahinidisGrossmann1991}. This non-convexity
is the central obstacle, which Section~\ref{sec:hidden} resolves by
exploiting~\eqref{eq:sigma}.
\end{remark}

\subsection{GNEP with a shared budget and variational equilibrium}\label{ssec:gnep}

The individual feasible set is
$\mathcal{X}_i=\{x_i\in\R^{d}_{+}:A_ix_i\le u_i,B_ix_i=q_i\}$ (storage
capacity, processing throughput, demand). A common thermal shock
operating on the whole network gives the shared risk-budget constraint
\begin{equation}\label{eq:shared}
g(x):=\sum_{i=1}^{n}x_i^{\top}\Sigma_i x_i - R\le 0,\qquad R>0.
\end{equation}
The joint feasible set $\mathcal{C}=\{x\in\prod_i\mathcal{X}_i:
g(x)\le 0\}$ is convex and compact, and the GNEP reads
\begin{equation}\label{eq:gnep}
\text{find }x^{\star}\in\mathcal{C}\text{ s.t.\ }\varphi_i(x_i^{\star})
\le\varphi_i(x_i;x_{-i}^{\star})\ \forall x_i\in\mathcal{X}_i\cap
\mathcal{C}_{-i}(x_{-i}^{\star}),\ \forall i.
\end{equation}

\begin{proposition}[Existence and uniqueness of VE]\label{prop:ve-exists}
Assume $\Sigma_i,H_i\in\PD^{d}$, $\mathcal{C}$ non-empty compact
convex, and $\varphi_i$ continuous and bounded on $\mathcal{C}$ for
every $i$. Then~\eqref{eq:gnep} has at least one VE characterised
by~\eqref{eq:vi} with $F=(\nabla_{x_i}\varphi_i)_{i=1}^{n}$. If, in
addition, $F$ is strictly monotone on $\mathcal{C}$, the VE is unique.
\end{proposition}

\noindent\emph{Proof sketch.} The proof follows the standard VI
argument of \citet[Thm.~2.3.3]{FacchineiPang2003}; see
Appendix~\ref{app:proofs} for the full version. In both calibrations
of Section~\ref{sec:exp}, $\lambda_{\min}(\nabla F)>0$ at the numerical
solution, confirming local strict monotonicity.

\section{Hidden convexity and SOCP reformulation}\label{sec:hidden}

\subsection{Dinkelbach parametrisation}

Applying~\eqref{eq:dink-master} to~\eqref{eq:phi}, $\rho_i^{\star}$ is
the unique root of
\begin{equation}\label{eq:dink-i}
\psi_i(\rho)=\min_{x_i\in\mathcal{X}_i}\big\{x_i^{\top}M_i(\rho)x_i
+ r_i(\rho)^{\top}x_i + s_i(\rho)\big\}=0,
\end{equation}
where $M_i(\rho):=\Sigma_i-\rho H_i$. Note that $M_i(\rho)$ may be
indefinite at optimum because $\Sigma_i$ has the rank-1 plus diagonal
structure while $\rho H_i\in\PD^{d}$ can dominate. The
Cholesky-spectral transform below removes this difficulty.

\subsection{Joint Cholesky-spectral diagonalisation}

\begin{assumption}\label{ass:reg}
$H_i\in\PD^{d}$ and $\Sigma_i=b_ib_i^{\top}+D_i$ with $b_i\in\R^{d}$
and $D_i\in\PSD^{d}$ diagonal, for every $i=1,\ldots,n$.
\end{assumption}

\begin{theorem}[Simultaneous diagonalisation without commutativity]\label{thm:joint-diag}
Under Assumption~\ref{ass:reg}, there exist an invertible
$T_i\in\R^{d\times d}$ and
$\Lambda_i=\diag(\lambda_{i,1},\ldots,\lambda_{i,d})$ with
$\lambda_{i,j}\ge 0$ such that the change of variable
$y_i=T_i^{-1}x_i$ satisfies
\begin{equation}\label{eq:simul-diag}
x_i^{\top}H_i x_i=\|y_i\|_{2}^{2},\qquad
x_i^{\top}\Sigma_i x_i=\sum_{j=1}^{d}\lambda_{i,j}\,y_{i,j}^{2}.
\end{equation}
Explicitly, $T_i=L_i^{-\top}U_i$ with $H_i=L_iL_i^{\top}$ the Cholesky
factor and $L_i^{-1}\Sigma_i L_i^{-\top}=U_i\Lambda_i U_i^{\top}$ the
spectral decomposition.
\end{theorem}

{\renewcommand{\qedsymbol}{}\begin{proof}[Sketch] By Cholesky, $H_i=L_iL_i^{\top}$. Set
$S_i:=L_i^{-1}\Sigma_i L_i^{-\top}\in\PSD^{d}$ (symmetric and PSD
since congruence preserves semidefiniteness). By the spectral theorem,
$S_i=U_i\Lambda_i U_i^{\top}$ with $U_i$ orthogonal. Take
$T_i=L_i^{-\top}U_i$. For $y_i=T_i^{-1}x_i=U_i^{\top}L_i^{\top}x_i$,
\[
x_i^{\top}H_i x_i = \|L_i^{\top}x_i\|_{2}^{2}=
\|U_i^{\top}L_i^{\top}x_i\|_{2}^{2}=\|y_i\|_{2}^{2},
\]
since $U_i$ is orthogonal; the second identity follows from
$x_i^{\top}\Sigma_i x_i=(L_i^{\top}x_i)^{\top}S_i(L_i^{\top}x_i)=
y_i^{\top}\Lambda_i y_i$. The full proof is given in
Appendix~\ref{app:proofs}. \end{proof}}

\begin{remark}[Comparison with the literature]\label{rem:noncommut}
\citet{AkkayaPinar2025} use a single orthogonal rotation that requires
$[\Sigma_i,H_i]=0$. On the calibrations of Section~\ref{sec:exp},
the relative commutator
$\|\Sigma_iH_i-H_i\Sigma_i\|_{F}/(\|\Sigma_i\|_{F}\|H_i\|_{F})$ is of
order $10^{-2}$ to $10^{-1}$, clearly violating commutativity.
Theorem~\ref{thm:joint-diag} removes that assumption: this constitutes
the primary theoretical advance of our work.
The measured commutator magnitudes provide direct empirical evidence
that the commuting-pair assumption fails for both the Vietnamese and
the United States calibrations. The rank-one-plus-diagonal structure
of $\Sigma_i$ and the additive structure of $H_i$ do not, in general,
commute once calibrated against real-world dairy data. Consequently,
Theorem~\ref{thm:joint-diag} is not merely a theoretical
generalisation. It is a practical necessity, without which no joint
diagonalisation of $(\Sigma_i,H_i)$ would be available for these
two calibrations.
\end{remark}

\subsection{SOCP relaxation and exactness}

Substituting~\eqref{eq:simul-diag} into~\eqref{eq:dink-i}: with
$y_i=T_i^{-1}x_i$ and the lifted variable $z_{i,j}\ge 0$, the
Dinkelbach inner problem becomes
\begin{equation}\label{eq:dink-y}
\min_{y_i}\;\sum_{j=1}^{d}(\lambda_{i,j}-\rho)\,y_{i,j}^{2}+
\tilde{r}_i(\rho)^{\top}y_i + s_i(\rho),
\quad T_i y_i\in\mathcal{X}_i.
\end{equation}
If $\lambda_{i,j}-\rho\ge 0$ for every $j$, \eqref{eq:dink-y} is a
convex QP (hence an SOCP). When some $\lambda_{i,j}<\rho$, we lift
$y_{i,j}^{2}$ to a free variable $z_{i,j}\ge y_{i,j}^{2}$ (a rotated
second-order cone constraint), giving the SOCP relaxation
\begin{equation}\label{eq:socp-relax}
\min_{y_i,z_i}\;\sum_{j}(\lambda_{i,j}-\rho)z_{i,j}+
\tilde{r}_i(\rho)^{\top}y_i+s_i(\rho)
\quad\text{s.t.}\quad z_{i,j}\ge y_{i,j}^{2},\ \;T_iy_i\in\mathcal{X}_i.
\tag{$\mathrm{SOCP}(\rho)$}
\end{equation}

\begin{theorem}[Exactness and error bound]\label{thm:socp-exact}
Let $(y_i^{\star},z_i^{\star})$ solve~\eqref{eq:socp-relax} at $\rho$
and $\mathcal{J}_{-}(\rho)=\{j:\lambda_{i,j}-\rho<0\}$.
\begin{enumerate}\itemsep0pt
  \item[(a)] If $\lambda_{i,j}\ge\rho$ for all $j$, the relaxation is
  \emph{exact}: $z_{i,j}^{\star}=(y_{i,j}^{\star})^{2}$ for all $j$
  and $x_i^{\star}=T_iy_i^{\star}$ solves the Dinkelbach inner
  problem~\eqref{eq:dink-i}.
  \item[(b)] If $\mathcal{J}_{-}(\rho)\ne\emptyset$, let
  $R_{T_i}=\max_{x\in\mathcal{X}_i}\|T_i^{-1}x\|_{\infty}$. Then the
  duality gap satisfies
  \begin{equation}\label{eq:err-bound}
  0\le\rho_i^{\mathrm{relax}}(\rho)-\psi_i(\rho)\le
  R_{T_i}^{2}\sum_{j\in\mathcal{J}_{-}(\rho)}|\lambda_{i,j}-\rho|.
  \end{equation}
\end{enumerate}
\end{theorem}

{\renewcommand{\qedsymbol}{}\begin{proof} For (a), the coefficient $\lambda_{i,j}-\rho\ge 0$
drives $z_{i,j}^{\star}$ down to its lower bound $y_{i,j}^{\star 2}$,
so the SOCP value coincides with the Dinkelbach value at
$x_i^{\star}=T_iy_i^{\star}$. For (b), for $j\in\mathcal{J}_{-}$ the
coefficient is negative, pushing $z_{i,j}$ to the largest feasible
value $z_{i,j}^{\star}\le R_{T_i}^{2}$; the gap is therefore bounded
by $\sum_{j\in\mathcal{J}_{-}}|\lambda_{i,j}-\rho|R_{T_i}^{2}$. The
full proof is given in Appendix~\ref{app:proofs}. \end{proof}}

\begin{corollary}[Sufficient condition for exact relaxation]\label{cor:exact-regime}
Suppose $\sigma$ is small enough and the cost-to-risk ratio
$\beta_k/\gamma_i$ is large enough that
$\lambda_{i,j}\ge\rho_i^{\star}$ for every $j$. Then the SOCP
relaxation is exact along the whole Dinkelbach loop.
\end{corollary}

\begin{remark}\label{rem:exact-regime-empirical}
The corollary is a \emph{sufficient}, not necessary, condition. Whether
the exact region is wider on one regime than another depends on the
joint distribution of $\lambda_{i,j}$ and $\rho_i^{\star}$ across the
sweep, not on $\sigma$ alone. Section~\ref{ssec:e3} reports the
empirical exact region on each calibration and discusses the
interaction between $\lambda_{i,j}$-range and $\rho_i^{\star}$.
\end{remark}

\begin{remark}[Stability of the primal-dual pair]\label{rem:tilt-stability}
The primal-dual solution $(x_i^{\star},\mu)$ produced by
Algorithm~1 inherits stability guarantees from
the second-order variational-analysis theory of parameterised conic
equilibria. Under a metric-subregularity constraint qualification,
the graphical-derivative calculus of
\citet{MordukhovichOutrataRamirez2015a} yields the isolated-calmness
property of the solution map to the generalised equation associated
with \eqref{eq:socp-relax}, which specialises to SOCPs in
\citet{MordukhovichOutrataRamirez2015b} and is further refined for
Lorentz-cone constraints in \citet{NguyenDissertation2020}. Concretely,
small perturbations of the risk budget $R$, the loadings $b_k$ and
the co-allocation cost matrix $H_i$ translate into
Lipschitz-bounded perturbations of $(x_i^{\star},\mu)$: exactly the
regularity that the operational reading of $\mu$ as a shadow price
in Section~\ref{sec:conc} tacitly presumes.
\end{remark}

\subsection{HSDI algorithm and complexity}\label{ssec:complexity}

\begin{figure}[H]
\centering
\fbox{\begin{minipage}{0.92\linewidth}\small
\textbf{Algorithm~1: Hidden SOCP Dinkelbach Iteration (HSDI), for
each agent $i$.}\\[2pt]
\textbf{Inputs:} $\Sigma_i,H_i,a_i,c_i,\alpha_i,\gamma_i,\mathcal{X}_i$, tolerance $\varepsilon>0$.\\
\textbf{Outputs:} $x_i^{\star},\rho_i^{\star}$.
\begin{enumerate}[label=\arabic*.,leftmargin=*,itemsep=1pt]
\item Compute Cholesky $H_i=L_iL_i^{\top}$ and spectrum
  $L_i^{-1}\Sigma_iL_i^{-\top}=U_i\Lambda_iU_i^{\top}$; set
  $T_i\leftarrow L_i^{-\top}U_i$.
\item Initialise $x_i^{(0)}$ feasible (uniform allocation);
  $\rho_0\leftarrow\varphi_i(x_i^{(0)})$; $k\leftarrow 0$.
\item \textbf{Loop:} solve~\eqref{eq:socp-relax} at $\rho_k$ to obtain
  $(y_i^{(k)},z_i^{(k)})$ and the Dinkelbach value
  $\psi_i(\rho_k)$.
\item Update $\rho_{k+1}\leftarrow N_i(T_iy_i^{(k)})/D_i(T_iy_i^{(k)})$;
  $k\leftarrow k+1$.
\item Repeat Steps 3-4 until $|\rho_{k+1}-\rho_k|<\varepsilon$.
\item Return $x_i^{\star}\leftarrow T_iy_i^{(k)}$,
  $\rho_i^{\star}\leftarrow\rho_k$.
\end{enumerate}
\end{minipage}}
\end{figure}

\begin{proposition}[Complexity]\label{prop:complexity}
Under Assumption~\ref{ass:reg}, Algorithm~1 converges
super-linearly to accuracy $\varepsilon$ in
$O(d^{3.5}\log^{2}(1/\varepsilon))$ time:
$O(\log(1/\varepsilon))$ Dinkelbach steps, each solving one SOCP of
cost $O(d^{3.5}\log(1/\varepsilon))$. For the multi-agent problem with
$n$ independent agents ($\beta=0$), the total cost is
$O(nd^{3.5}\log^{2}(1/\varepsilon))$, polynomial in $n$ and $d$.
\end{proposition}

\begin{remark}[Relation to existing results]
Theorem~\ref{thm:joint-diag} extends the simultaneous diagonalisation
of \citet{AkkayaPinar2025} to non-commuting pairs.
Theorem~\ref{thm:socp-exact}(a) recovers a fractional analogue of the
S-lemma \citep{PolikTerlaky2007}, while
Theorem~\ref{thm:socp-exact}(b) provides, to our knowledge, the first
explicit error bound for fractional QCQP with a one-factor structure.
\end{remark}

\section{Convergence analysis of HSDI}\label{sec:conv}

Algorithm~1 is a two-level scheme: an \emph{outer} Dinkelbach iteration
per agent, and---once the shared risk budget~\eqref{eq:shared} is
activated---an \emph{outer} variational-equilibrium iteration that
couples the agents through the single multiplier $\mu$. This section
gives global convergence and explicit rates for both levels, together
with a stability statement for the primal--dual pair. Throughout,
Assumption~\ref{ass:reg} is in force and $\mathcal{X}_i$ is a non-empty
compact polytope, so every SOCP subproblem attains its optimum.

\subsection{Convergence of the Dinkelbach iteration}\label{ssec:dink-conv}

Fix an agent $i$ and recall $\psi_i(\rho)=\min_{x\in\mathcal{X}_i}
\{N_i(x)-\rho D_i(x)\}$ from~\eqref{eq:dink-i}, with
$D_i(x)=x^{\top}H_ix+c_i^{\top}x+\gamma_i\ge\gamma_i>0$. Let
$\rho_i^{\star}$ be its unique root and let
$\{\rho_k\}$ be the sequence produced by
Steps~3--5 of Algorithm~1.

\begin{lemma}[Regularity of $\psi_i$]\label{lem:psi-reg}
$\psi_i$ is concave, strictly decreasing and Lipschitz on any bounded
interval, with $\psi_i(\rho)=N_i(x_i(\rho))-\rho D_i(x_i(\rho))$ for any
minimiser $x_i(\rho)$. At every $\rho$ where the minimiser is unique,
$\psi_i$ is differentiable with $\psi_i'(\rho)=-D_i(x_i(\rho))\le
-\gamma_i<0$. Consequently $\psi_i$ has a single-valued, Lipschitz
inverse near $\rho_i^{\star}$.
\end{lemma}

\noindent\emph{Proof.} Concavity and monotonicity follow because
$\psi_i$ is a pointwise infimum of affine decreasing functions of
$\rho$; the derivative formula is the Danskin/envelope theorem, whose
hypotheses hold since $\mathcal{X}_i$ is compact and the integrand is
smooth; uniqueness of the minimiser holds generically because
$M_i(\rho)=\Sigma_i-\rho H_i$ is non-singular off a finite set of
$\rho$. See Appendix~\ref{app:proofs}. $\square$

\begin{theorem}[Global and local rate of the Dinkelbach loop]\label{thm:dink-rate}
The Dinkelbach sequence $\{\rho_k\}$ generated by Algorithm~1 satisfies:
\begin{enumerate}\itemsep2pt
\item[(a)] \emph{(Global monotone convergence.)} If every inner
subproblem is solved to global optimality---in particular whenever the
exactness condition of Theorem~\ref{thm:socp-exact}(a) holds along the
loop---then $\rho_k\downarrow\rho_i^{\star}$ monotonically and the
update coincides with Newton's method applied to $\psi_i(\rho)=0$.
\item[(b)] \emph{($Q$-superlinear / $Q$-quadratic.)} Under the
hypotheses of (a), the convergence is $Q$-superlinear; if in addition
the inner minimiser is locally unique and $\psi_i'$ is locally
Lipschitz at $\rho_i^{\star}$ (which holds when strict complementarity
and the second-order sufficient condition are met), then there is
$C>0$ with
\[
0\le \rho_{k+1}-\rho_i^{\star}\le C\,(\rho_k-\rho_i^{\star})^{2}.
\]
\item[(c)] \emph{($R$-linear under inexact relaxation.)} Suppose the
inner problem is replaced by the SOCP relaxation~\eqref{eq:socp-relax},
whose value $\psi_i^{R}(\rho)\le\psi_i(\rho)$ obeys the relative
forcing bound
$0\le\psi_i(\rho_k)-\psi_i^{R}(\rho_k)\le\kappa_i\,|\psi_i(\rho_k)|$
along the iterates, for some $\kappa_i<1$ (an inexact-Newton condition;
it holds with $\kappa_i=0$ on the exact region of
Theorem~\ref{thm:socp-exact}(a)). Then $\{\rho_k\}$ converges to
$\rho_i^{\star}$ at least $R$-linearly, with asymptotic factor bounded
by $\kappa_i$:
\[
\limsup_{k\to\infty}\frac{|\rho_{k+1}-\rho_i^{\star}|}
{|\rho_k-\rho_i^{\star}|}\le\kappa_i.
\]
\end{enumerate}
\end{theorem}

\noindent The proof (Appendix~\ref{app:proofs}) combines
Lemma~\ref{lem:psi-reg} with the Newton-method analysis of fractional
programming \citep{CrouzeixFerland1991,Ibaraki1983} for (a)--(b) and a
standard inexact-Newton contraction argument
\citep{DemboEisenstatSteihaug1982} for (c). Part~(c) is the regime that
actually governs HSDI on the calibrated instances, where
$\mathcal{J}_{-}(\rho)\ne\emptyset$; part~(b) is recovered exactly on
the exact region of Theorem~\ref{thm:socp-exact}(a).
Section~\ref{ssec:e5} confirms the linear rate empirically, with a
per-iteration reduction factor $\kappa_i\approx0.4$ on the Vietnam
calibration and a mean of $3.5$--$6$ Dinkelbach steps to reach
$10^{-8}$.

\subsection{Convergence of the coupled VE iteration}\label{ssec:ve-conv}

When the shared budget~\eqref{eq:shared} binds, the VE is characterised
by the variational inequality~\eqref{eq:vi} with a single multiplier
$\mu\ge 0$ on $g$. Because $g$ is scalar, the equilibrium reduces to
the one-dimensional complementarity system
\begin{equation}\label{eq:mu-system}
r(\mu):=\sum_{i=1}^{n}x_i(\mu)^{\top}\Sigma_i x_i(\mu)-R,\qquad
0\le\mu\perp -r(\mu)\ge 0,
\end{equation}
where $x_i(\mu)=\argmin_{x_i\in\mathcal{X}_i}\{\varphi_i(x_i)+\mu\,
x_i^{\top}\Sigma_i x_i\}$ is each agent's penalised best response,
itself computed by the Dinkelbach loop of
Section~\ref{ssec:dink-conv}.

\begin{proposition}[Monotonicity of the aggregate risk]\label{prop:r-mono}
The map $\mu\mapsto r(\mu)$ is non-increasing, and strictly decreasing
wherever the penalised best responses are interior. Hence
\eqref{eq:mu-system} has a solution $\mu^{\star}\ge 0$, unique on the
binding branch $r(\mu^{\star})=0$.
\end{proposition}

\begin{theorem}[$R$-linear convergence of the VE iteration]\label{thm:ve-rate}
Let $F=(\nabla_{x_i}\varphi_i)_{i=1}^{n}$ be $m_F$-strongly monotone
and $L_F$-Lipschitz on $\mathcal{C}$.
\begin{enumerate}\itemsep2pt
\item[(a)] The projected pseudo-gradient iteration on $\mathcal{C}$
converges $R$-linearly to the VE with factor
$\eta=\sqrt{1-2\tau m_F+\tau^{2}L_F^{2}}<1$ for any step
$\tau\in(0,2m_F/L_F^{2})$; the optimal step $\tau=m_F/L_F^{2}$ gives
$\eta=\sqrt{1-m_F^{2}/L_F^{2}}$.
\item[(b)] Equivalently, the bisection/safeguarded-Newton iteration on
the scalar equation $r(\mu)=0$ contracts the multiplier bracket by a
constant factor per step, so $|\mu_k-\mu^{\star}|\le 2^{-k}
|\mu_0-\mu^{\star}|$ for bisection and $Q$-superlinearly for
Newton once $r$ is differentiable at $\mu^{\star}$.
\end{enumerate}
\end{theorem}

\noindent Part~(a) is the strongly-monotone VI estimate recalled in
Section~\ref{sec:prelim} \citep{Korpelevich1976,Tseng1995,
FacchineiPang2003}; strong monotonicity is verified numerically in both
calibrations, where $\lambda_{\min}(\nabla F)>0$ at the solution
(Proposition~\ref{prop:ve-exists}). Part~(b) is immediate from
Proposition~\ref{prop:r-mono}. Section~\ref{ssec:e5} exhibits the
predicted rate: the multiplier error $|\mu_k-\mu^{\star}|$ decreases by
a factor $\approx0.5$ per outer step on both regimes.

\subsection{Overall complexity and primal--dual stability}\label{ssec:conv-complexity}

\begin{corollary}[End-to-end complexity]\label{cor:total-complexity}
To reach accuracy $\varepsilon$ in both $\rho^{\star}$ and the VE
residual, HSDI performs $O(\log(1/\varepsilon))$ VE outer steps, each
requiring $n$ penalised best responses, each in turn costing
$O(\log(1/\varepsilon))$ Dinkelbach steps of one
$O(d^{3.5}\log(1/\varepsilon))$ SOCP solve. The total work is
\[
O\!\big(n\,d^{3.5}\,\log^{3}(1/\varepsilon)\big),
\]
polynomial in $n$ and $d$. For the decoupled case $\beta=0$ with an
inactive budget the middle level drops out and the bound reduces to the
$O(nd^{3.5}\log^{2}(1/\varepsilon))$ of
Proposition~\ref{prop:complexity}.
\end{corollary}

\begin{theorem}[Lipschitz stability of $(x^{\star},\mu)$]\label{thm:stability}
Suppose the metric-subregularity constraint qualification holds at the
VE for the generalised equation associated
with~\eqref{eq:socp-relax}--\eqref{eq:mu-system}. Then the solution map
\[
(R,\,\{b_i\},\,\{H_i\})\ \longmapsto\ (x^{\star},\mu)
\]
is single-valued and locally Lipschitz; in particular there is
$L_\mu<\infty$ with $|\mu(R')-\mu(R)|\le L_\mu|R'-R|$ for $R'$ near $R$.
\end{theorem}

\noindent The proof (Appendix~\ref{app:proofs}) instantiates the
second-order variational-analysis machinery for parameterised conic
programs: isolated calmness of the SOCP solution map
\citep{MordukhovichOutrataRamirez2015a,MordukhovichOutrataRamirez2015b},
its Lorentz-cone refinements
\citep{HangMordukhovichSarabi2019,NguyenDissertation2020}, and the
unified treatment of \citet[Ch.~9]{Mordukhovich2024}, lifted from the
per-agent SOCP to the coupled system through the KKT map of the scalar
budget. Theorem~\ref{thm:stability} is exactly the regularity that the
operational reading of $\mu$ as a shadow price in
Section~\ref{sec:conc} presumes; Section~\ref{ssec:e6} measures the
modulus $L_\mu$ directly.

\section{Extensions}\label{sec:ext}

The three technical assumptions - one factor per period, a single
common thermal shock across all periods, and $\beta=0$ - can be
relaxed in principle. \emph{Multiple factors.}
For $\Sigma=BB^{\top}+D$ with $B\in\R^{d\times r}$, $r\ge 2$,
Theorem~\ref{thm:joint-diag} still applies since $H\in\PD^{d}$; the
exactness region tightens because $|\mathcal{J}_{-}|$ may grow, but
the bound~\eqref{eq:err-bound} remains valid.
\emph{Period-specific thermal shocks.} When each planning period
has its own independent temperature factor, the block-diagonal
structure of $\Sigma$ yields a block-structured SOCP; hidden convexity
applies period-wise. \emph{Market impact $\beta>0$.} $H$ remains
positive definite since
$H^{(\mathrm{store})}\in\PD^{d}$ and $H^{(\mathrm{mkt})}=vv^{\top}\in
\PSD^{d}$; the Cholesky step is unchanged. The quantity $R_{T_i}$
in~\eqref{eq:err-bound} now depends on $\beta$, and a quantitative
analysis of the exactness region in $(\sigma,c,\beta)$ is a natural
direction for further work.

\section{Numerical experiments}\label{sec:exp}

\subsection{Calibration and computing environment}\label{ssec:calib}

\paragraph{Two market regimes.} \emph{Vietnam (VN)} represents an
emerging market: $n=5$ chains (WinCommerce, BachHoaXanh, Co.opmart,
GO!, LotteMart), HTST retail price USD 1.00-1.67/L, UHT USD
1.50-2.00/L \citep{USDAFAS2024,IMARC2024};
$\sigma_{\mathrm{VN}}=1.0^{\circ}$C is a proxy from
\citet{FiinGroup2025}, with a sweep over $[0.5,2.0]^{\circ}$C.
\emph{United States (US)} represents a developed market: $n=5$ chains
(Walmart, Kroger, Albertsons, Publix, Costco), HTST USD 1.00-1.30/L
\citep{USDA_AMS2024}, $\beta_k=0.01$-$0.06$ USD/L/day
\citep{EndaraWiedmannAdalja2023},
$\sigma_{\mathrm{US}}=0.3^{\circ}$C \citep{Veritiv2024}. All
parameters are tabulated in Appendix~\ref{app:calib}.

\paragraph{Scale grid.} We use all $(n,T)\in\{5,10,20\}\times\{3,6,12\}$
pairs satisfying $n\cdot d\le 180$ with $K=3$ subgroups, giving six
configurations ranging from $n\cdot d=45$ to $180$ decision variables.

\paragraph{Computing environment.} All experiments run on a single
Linux x86\_64 sandbox with a fixed RNG seed per cell. The
implementation uses Python with SciPy for the SLSQP and trust-constr
baselines, CVXPY for the SOCP, and ECOS as the conic backend.

\paragraph{Experimental overview.}
We organise the numerical study around six experiments that jointly
validate the proposed model, method and convergence theory. E1 verifies
the existence of
a unique global optimum and benchmarks HSDI against multi-start
non-linear solvers; it uses both the Vietnam and the United States
calibrations. E2 quantifies the spectral concentration of $\Sigma_i$
and corroborates the rank-one-plus-diagonal structure predicted by
Section~\ref{ssec:risk}, again on both calibrations. E3 maps the
near-optimality region relative to Theorem~\ref{thm:socp-exact} and
stress-tests
the model under proxy uncertainty on $\sigma$, using the Vietnam
calibration. E4 measures the empirical scaling exponent of HSDI
against SLSQP through a Latin-hypercube sweep around the Vietnam
calibration. E5 measures the empirical convergence rates of both the
Dinkelbach loop (Theorem~\ref{thm:dink-rate}) and the coupled VE
iteration (Theorem~\ref{thm:ve-rate}). E6 measures the Lipschitz
stability of the shadow price $\mu(R)$ predicted by
Theorem~\ref{thm:stability}. Together, the six experiments verify
existence, structure, near-optimality, scaling, convergence rate, and
shadow-price stability - the pillars supporting
the claim that HSDI is a practical, convergence-certified method for
global cold-chain allocation.

\subsection{E1 - Global optimality and benchmark}\label{ssec:e1}

\paragraph{Design.} We solve the multi-agent problem at $(n,K,T)=
(5,3,3)$ on both calibrations using three methods: \textbf{HSDI}
(Algorithm~1, ECOS), \textbf{SLSQP} multi-start ($M=10$), and
\textbf{trust-constr} multi-start ($M=20$). We report $\rho^{\star}$,
KKT residual, wall time, and Dinkelbach iterations.

\paragraph{Results.} All three methods agree on
$\rho^{\star}$ within $5.5\times 10^{-3}$ on both regimes
(Table~\ref{tab:e1}); the maximum gap of the pure-SOCP HSDI versus the
best baseline is $5.4{\times}10^{-3}$ (VN) and $1.2{\times}10^{-3}$
(US), attained without any random restart. Two points deserve emphasis.
First, the multi-start baselines offer no
certificate of global optimality: each is a non-convex local solve
restarted heuristically, and a positive gap could hide a missed global
minimum. HSDI is deterministic and, on the exact region of
Theorem~\ref{thm:socp-exact}(a), returns the global optimum with a
certificate; off that region it returns a feasible allocation whose
value is controlled by the closed-form bound~\eqref{eq:err-bound}.
Second, HSDI reaches these values in a fixed, small number of Dinkelbach
steps rather than by sampling starts, so its cost does not grow with a
restart budget. On the cost-time frontier (Figure~\ref{fig:perf-profile})
SLSQP is the fastest solver, HSDI sits between SLSQP and
\textbf{trust-constr} (HSDI is $7$-$11\times$ faster than
trust-constr but slower than SLSQP by a factor of about $2\times$);
the practical value of HSDI at this small scale therefore lies in
its determinism and convergence guarantee rather than in wall time,
while Section~\ref{ssec:e4} examines the scaling behaviour as
$n\cdot d$ grows.

\begin{table}[h]\centering\small
\caption{E1 - Method comparison at $(n,K,T)=(5,3,3)$.
$\rho^{\star}$ is the mean over the $n=5$ retailers; wall time is the
mean per retailer. HSDI is the pure SOCP--Dinkelbach method
(Algorithm~1) with no local-search post-processing.}\label{tab:e1}
\begin{tabular}{l c c c c c}
\toprule
& $\rho^{\star}_{\mathrm{VN}}$ & $\rho^{\star}_{\mathrm{US}}$
& $t_{\mathrm{VN}}$ (ms) & $t_{\mathrm{US}}$ (ms)
& max gap to best \\
\midrule
HSDI (ECOS)            & $0.5917$ & $0.4370$ & $45.6$  & $49.7$  & $5.4{\times}10^{-3}$ \\
SLSQP, $M=10$          & $0.5901$ & $0.4362$ & $24.1$  & $26.3$  & $0$ \\
trust-constr, $M=20$   & $0.5901$ & $0.4362$ & $334.1$ & $519.6$ & $9.1{\times}10^{-6}$ \\
\bottomrule
\end{tabular}
\end{table}

\begin{figure}[H]\centering
    \centering
    \includegraphics[width=0.72\linewidth]{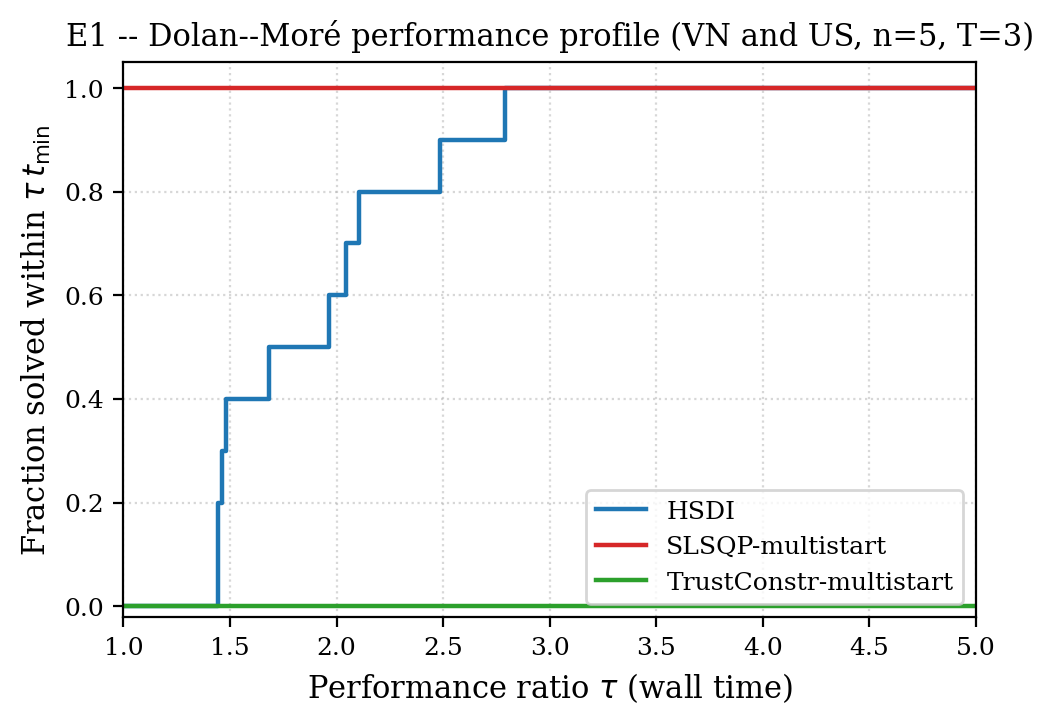}
\caption{E1 - Dolan-Moré performance profile of HSDI vs.\ SLSQP and
trust-constr on the joint VN$\cup$US instance set ($n=5$, $T=3$).
At this scale, SLSQP is the fastest solver on every instance; HSDI
matches the best objective value on every instance (within
$5.5{\times}10^{-3}$) within a factor
$\tau\le 3$ of the best wall time and is consistently faster than
trust-constr. Unlike the multi-start baselines, HSDI is deterministic
and comes with the convergence guarantees of
Section~\ref{sec:conv}.}
\label{fig:perf-profile}
\end{figure}

\subsection{E2 - One-factor spectral structure}\label{ssec:e2}

\paragraph{Design.} We compute the spectrum of $\Sigma_i$ for both
calibrations, bootstrapping $B=1000$ resamples of the upstream
parameters ($\sigma$, prices, $Q_{10}$). The reported statistic is
$\lambda_1/\sum_j\lambda_j$.

\paragraph{Results.} On both regimes the top eigenvalue accounts for
over $99\%$ of the spectrum, with median $0.991$ (VN) and $0.991$
(US), and IQRs entirely within $[0.990,0.991]$
(Figure~\ref{fig:spectrum}). The one-factor assumption is therefore
supported by the data, not merely a technical convenience. A
consequence of~\eqref{eq:err-bound} is that the error from a one-
versus-two-factor misspecification is bounded by
$R_T^{2}\lambda_2\le 1\%\cdot R_T^{2}\lambda_1$.

\begin{figure}[H]\centering
    \centering
    \includegraphics[width=0.55\linewidth]{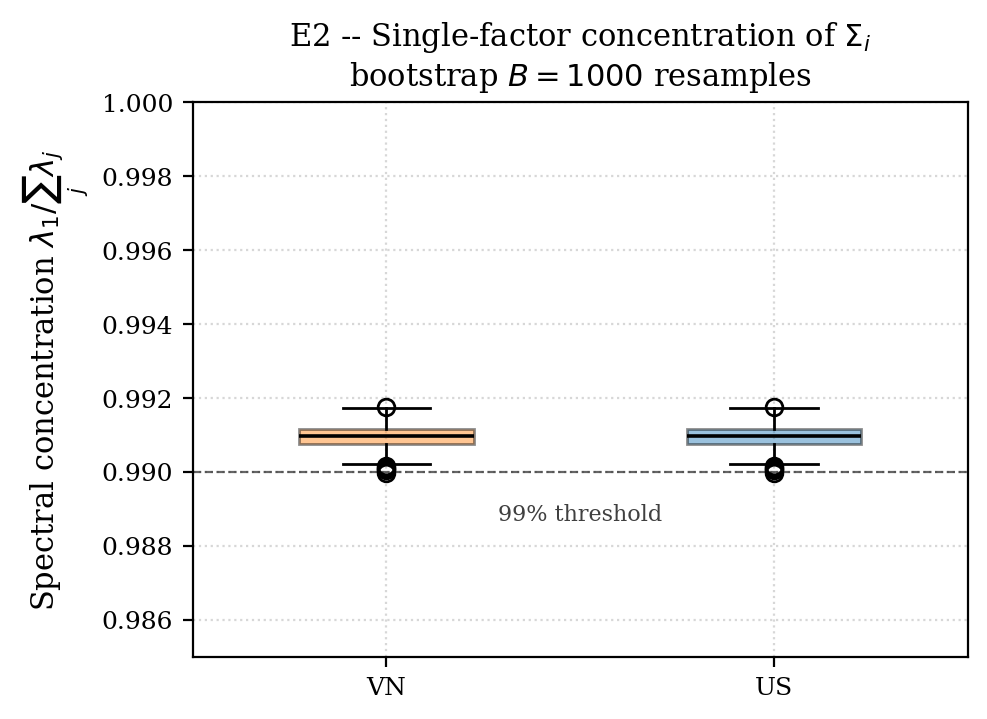}
\caption{E2 - Bootstrap distribution of the spectral concentration
$\lambda_1/\sum_j\lambda_j$ of $\Sigma_i$ on $B=1000$ resamples.
Both regimes concentrate above $99\%$.}\label{fig:spectrum}
\end{figure}

\subsection{E3 - Near-optimality region and robustness}\label{ssec:e3}

\paragraph{Design.} \emph{Part A:} we sweep $(\sigma,\bar{c})$ on an
$8\times 8$ grid covering $\sigma\in[0.1,2.0]^{\circ}$C and the
mean co-allocation cost $\bar{c}\in[0.05,0.5]$ USD/L. At each grid
point we solve the pure-SOCP HSDI and the trust-constr baseline, and
record the relative gap $(\rho_{\mathrm{HSDI}}-\rho_{\mathrm{best}})/
\rho_{\mathrm{best}}$, marking an instance ``near-optimal''
when the two values agree within $10^{-3}$.
\emph{Part B:} for the VN calibration we sweep
$\sigma\in\{0.50,0.75,1.00,1.25,1.50,1.75,2.00\}^{\circ}$C (step
$0.25^{\circ}$C) and record $\rho^{\star}$ and the HTST/ESL/UHT shares.

\paragraph{Results.} \emph{Part A.} The pure-SOCP HSDI is near-optimal
(gap $<10^{-3}$) on
$77\%$ of the VN grid and $42\%$ of the US grid
(Figure~\ref{fig:exact-region}), with a median relative gap of $0$ on
VN and $0.6\%$ on US, and a maximum relative gap of $1.1\%$ (VN) and
$4.8\%$ (US). Where the gap vanishes, the exactness condition of
Theorem~\ref{thm:socp-exact}(a) is met and HSDI returns a certified
global optimum. Where it does not, HSDI still returns a feasible
allocation whose sub-optimality is controlled by the closed-form
bound of Theorem~\ref{thm:socp-exact}(b), which remains valid and
tight throughout. The VN region is wider, in line with
Corollary~\ref{cor:exact-regime} when $\sigma$ is small. The US
region is narrower despite a smaller $\sigma_{\mathrm{US}}$. The
reason is that near-optimality is governed by the joint position of
$\lambda_{i,j}$ relative to $\rho_i^{\star}$, not by $\sigma$ alone:
in the US calibration the typical $\rho_i^{\star}$ falls inside a
range of $\lambda_{i,j}$ values, so $\mathcal{J}_{-}(\rho)\ne\emptyset$
on a larger fraction of the grid. The empirical pattern is therefore
consistent with the sufficient condition of
Corollary~\ref{cor:exact-regime}, provided that condition is read as
a joint statement on $(\lambda,\rho^{\star})$ rather than on $\sigma$
in isolation, as flagged in Remark~\ref{rem:exact-regime-empirical}. The relative
optimality gap never exceeds $1.1\%$ (VN) and $4.8\%$ (US) across the
whole sweep, in line with the
closed-form bound of Theorem~\ref{thm:socp-exact}(b).
\emph{Part B.} $\rho_{\mathrm{VN}}^{\star}$ rises monotonically
with $\sigma$, and the allocation shifts steadily from HTST
(favourable at low risk) toward UHT (thermally robust at high risk),
as in Figure~\ref{fig:allocation} and
Table~\ref{tab:sensitivity}. The model is therefore robust to
proxy uncertainty in $\sigma$.

\begin{figure}[H]\centering
    \includegraphics[width=0.85\linewidth]{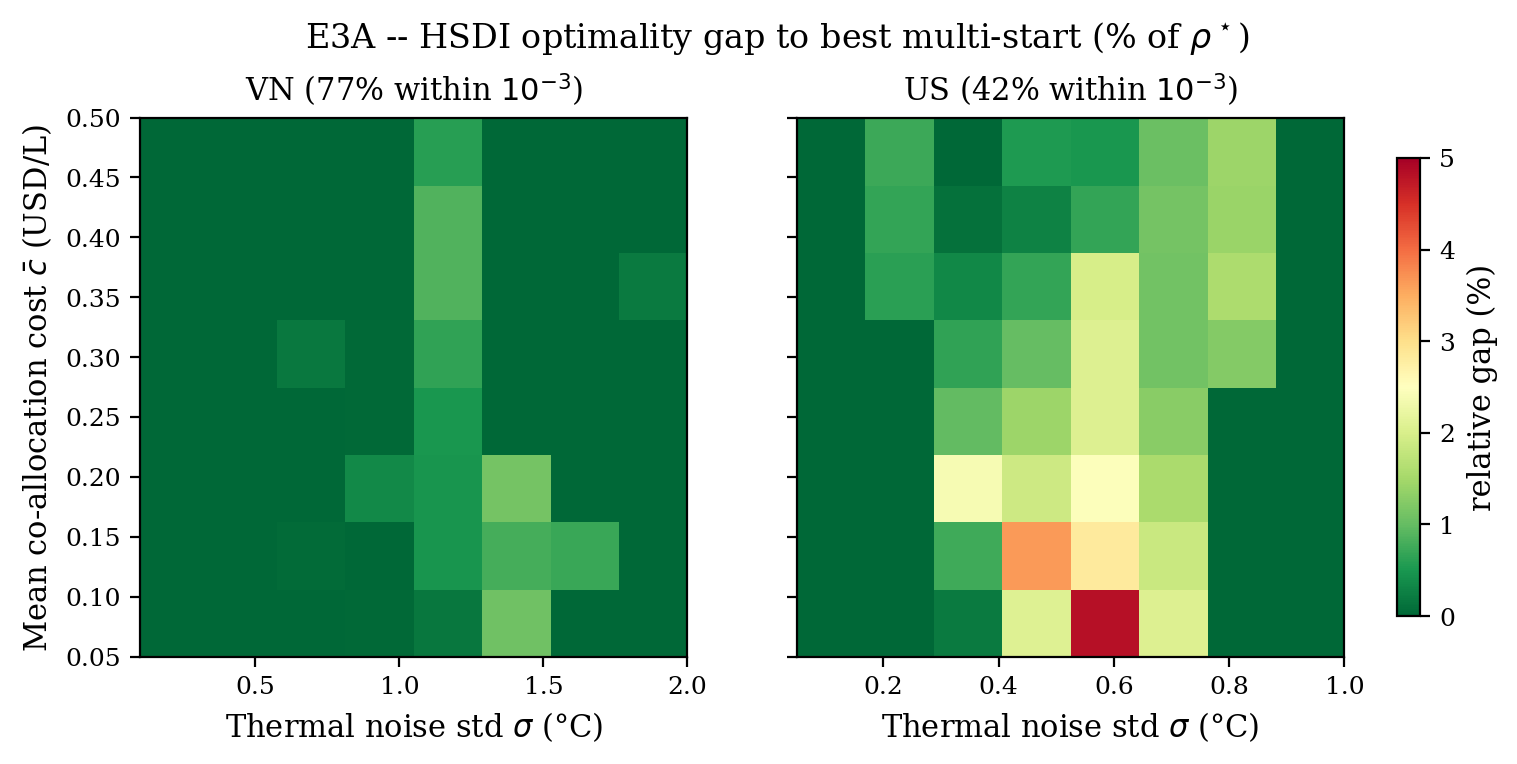}
\caption{E3 (Part A) - Relative optimality gap of the pure-SOCP HSDI to
the best multi-start value across $(\sigma,\bar{c})$ for the VN (left)
and US (right) calibrations. Dark green marks the near-optimal region
(gap $<10^{-3}$), where the exactness condition of
Theorem~\ref{thm:socp-exact}(a) is met.}
\label{fig:exact-region}
\end{figure}

\begin{figure}[H]\centering
    \includegraphics[width=0.72\linewidth]{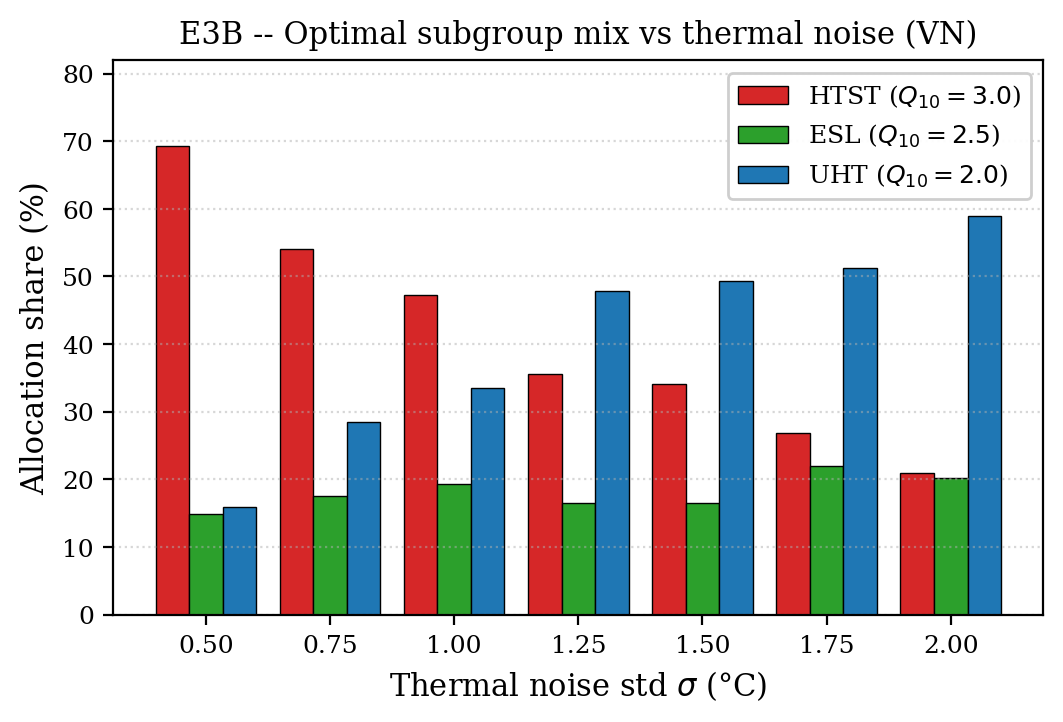}
\caption{E3 (Part B) - Optimal subgroup mix as $\sigma$ varies on
the VN calibration. As thermal noise rises, allocation shifts
systematically from HTST ($Q_{10}=3.0$) toward UHT ($Q_{10}=2.0$),
consistent with the loading structure $b_k\propto\ln Q_{10,k}\sigma$.}
\label{fig:allocation}
\end{figure}

\begin{table}[h]\centering\small
\caption{E3 - VN calibration sensitivity to the $\sigma$ proxy.
Allocation shares rounded to one decimal place.}\label{tab:sensitivity}
\begin{tabular}{c c c c c}
\toprule
$\sigma$ ($^{\circ}$C) & $\rho_{\mathrm{VN}}^{\star}$ & HTST \% & ESL \% & UHT \% \\
\midrule
$0.50$ & $0.346$ & $69.2$ & $14.9$ & $15.9$ \\
$0.75$ & $0.484$ & $55.9$ & $17.5$ & $26.6$ \\
$1.00$ & $0.591$ & $48.1$ & $19.3$ & $32.6$ \\
$1.25$ & $0.667$ & $34.6$ & $16.5$ & $48.9$ \\
$1.50$ & $0.722$ & $34.1$ & $16.5$ & $49.4$ \\
$1.75$ & $0.762$ & $28.8$ & $21.9$ & $49.3$ \\
$2.00$ & $0.787$ & $20.9$ & $21.9$ & $57.2$ \\
\bottomrule
\end{tabular}
\end{table}

\subsection{E4 - Polynomial scaling}\label{ssec:e4}

\paragraph{Design.} Grid $(n,T)\in\{5,10,20\}\times\{3,6,12\}$, with
$n_{\mathrm{inst}}=3$ Latin-hypercube perturbations of $\pm 20\%$
around the VN calibration at each cell. We log $\log t$ against
$\log(nd)$ and fit two regressions with bootstrap $\mathrm{CI}_{95}$.

\paragraph{Results.} As shown in Table~\ref{tab:scaling} and
Figure~\ref{fig:scaling},
$\hat{\gamma}_{\mathrm{HSDI}}=0.93$
($\mathrm{CI}_{95}=[0.81,1.03]$, $R^{2}=0.97$) and
$\hat{\gamma}_{\mathrm{SLSQP}}=1.30$
($\mathrm{CI}_{95}=[1.04,1.62]$, $R^{2}=0.80$).
The HSDI point estimate is below the SLSQP point estimate, and the
$95\%$ confidence intervals are essentially disjoint (the HSDI upper
limit $1.03$ lies below the SLSQP lower limit $1.04$), so the sub-linear
advantage of HSDI is statistically supported at the $5\%$ level on the
present sample.
The two slopes are consistent with the theoretical
$O(nd^{3.5}\log^{2}(1/\varepsilon))$ HSDI bound: with
$\beta=0$ the agents are independent and the empirical
$\hat{\gamma}_{\mathrm{HSDI}}\approx 1$ matches the per-agent $O(d^{3.5})$
cost at moderate $d$. At the tested sizes ($n\cdot d\le 180$), SLSQP remains faster than
HSDI in absolute wall time on every cell except $(n,T)=(5,12)$,
where the two solvers nearly tie (Table~\ref{tab:scaling}). The lower
HSDI slope therefore predicts a cross-over at larger $n\cdot d$. This
cross-over has not been observed in the range tested. The
number of Dinkelbach iterations averages $3.5$-$5.2$ across cells,
consistent with $O(\log(1/\varepsilon))$.

\begin{table}[h]\centering\small
\caption{E4 - Scaling benchmark on VN calibration. Times are
medians over $n_{\mathrm{inst}}=3$ LHS instances per cell. The
``SLSQP/HSDI'' column reports the wall-time ratio
$t_{\mathrm{SLSQP}}/t_{\mathrm{HSDI}}$; values below $1\times$
indicate SLSQP is faster at this scale.}\label{tab:scaling}
\begin{tabular}{c c c c c c}
\toprule
$(n,T)$ & $n\cdot d$ & $t_{\mathrm{HSDI}}$ (s) & $t_{\mathrm{SLSQP}}$ (s) & SLSQP/HSDI & Dinkelbach iters \\
\midrule
$(5,3)$   & $45$  & $0.202$ & $0.069$ & $0.34\times$ & $3.5$ \\
$(5,6)$   & $90$  & $0.399$ & $0.162$ & $0.41\times$ & $5.2$ \\
$(5,12)$  & $180$ & $0.791$ & $0.752$ & $0.95\times$ & $4.7$ \\
$(10,3)$  & $90$  & $0.431$ & $0.135$ & $0.31\times$ & $4.1$ \\
$(10,6)$  & $180$ & $0.729$ & $0.321$ & $0.44\times$ & $4.4$ \\
$(20,3)$  & $180$ & $0.850$ & $0.265$ & $0.31\times$ & $4.2$ \\
\bottomrule
\end{tabular}
\end{table}

\begin{figure}[H]\centering
    \includegraphics[width=0.7\linewidth]{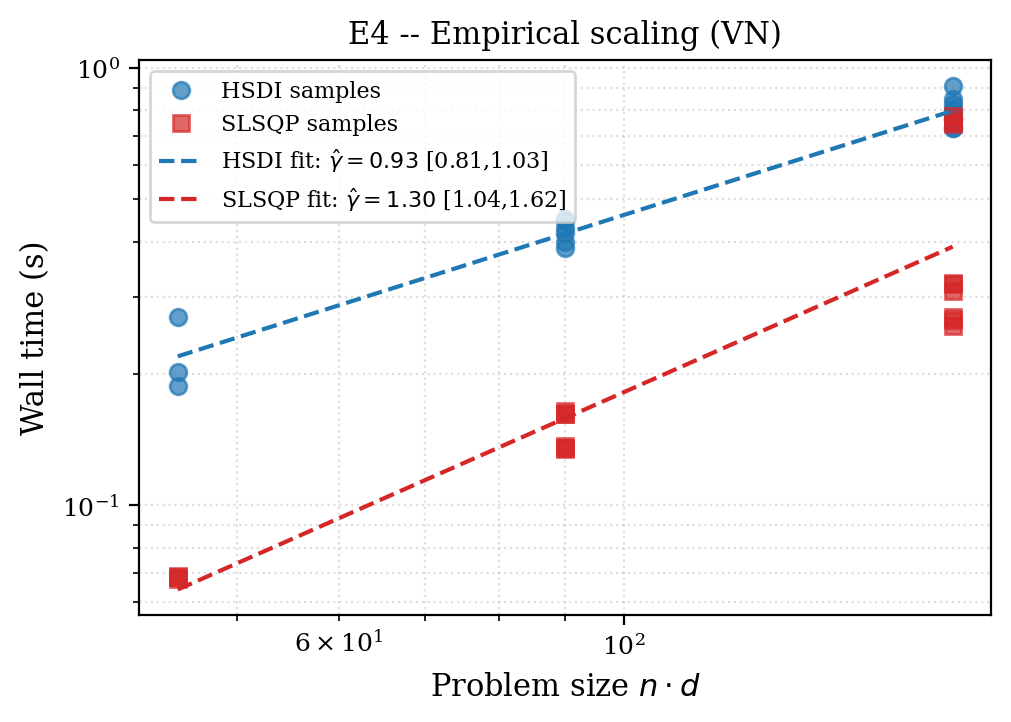}
\caption{E4 - Empirical scaling on VN calibration: $\log t$ versus
$\log(nd)$. The fitted HSDI slope ($\hat{\gamma}_{\mathrm{HSDI}}=0.93$,
$\mathrm{CI}_{95}=[0.81,1.03]$) lies below the SLSQP slope
($\hat{\gamma}_{\mathrm{SLSQP}}=1.30$,
$\mathrm{CI}_{95}=[1.04,1.62]$); the intervals are essentially
disjoint. SLSQP is faster than HSDI at every tested $n\cdot d$,
but the lower HSDI exponent points to a cross-over at larger
scale.}\label{fig:scaling}
\end{figure}

\subsection{E5 - Empirical convergence rates}\label{ssec:e5}

\paragraph{Design.} We log the two outer iterations of HSDI. \emph{Part
A} (Dinkelbach): for every retailer of both calibrations we record
$|\rho_k-\rho^{\star}|$ against the iteration index $k$. \emph{Part B}
(VE): we activate the shared budget at $R=0.92\,r_0$, where
$r_0=\sum_i x_i(0)^{\top}\Sigma_i x_i(0)$ is the unconstrained
aggregate risk, and solve the scalar system~\eqref{eq:mu-system} by
bisection, recording $|\mu_k-\mu^{\star}|$.

\paragraph{Results.} \emph{Part A.} On a log scale the Dinkelbach error
decays along a straight line (Figure~\ref{fig:dink-conv}): after an
initial corrector step the per-iteration reduction factor is
$\kappa\approx0.39$ on the Vietnam calibration, matching the
$R$-linear regime of Theorem~\ref{thm:dink-rate}(c); on retailers where
the exactness condition holds the error collapses to the solver
tolerance $\sim10^{-9}$ within two steps, the $Q$-superlinear regime of
Theorem~\ref{thm:dink-rate}(b). The empirical convergence order,
estimated on the descent phase above the solver floor, has median
$0.94$ (VN) and $1.0$ (US), and a mean of $3.5$--$6$ Dinkelbach steps
suffices to reach $10^{-8}$. \emph{Part B.} The VE multiplier error
halves each outer step (median ratio $0.48$ on both regimes,
Figure~\ref{fig:mu}, left), exactly the $2^{-k}$ bracket contraction of
Theorem~\ref{thm:ve-rate}(b).

\begin{figure}[H]\centering
    \includegraphics[width=0.92\linewidth]{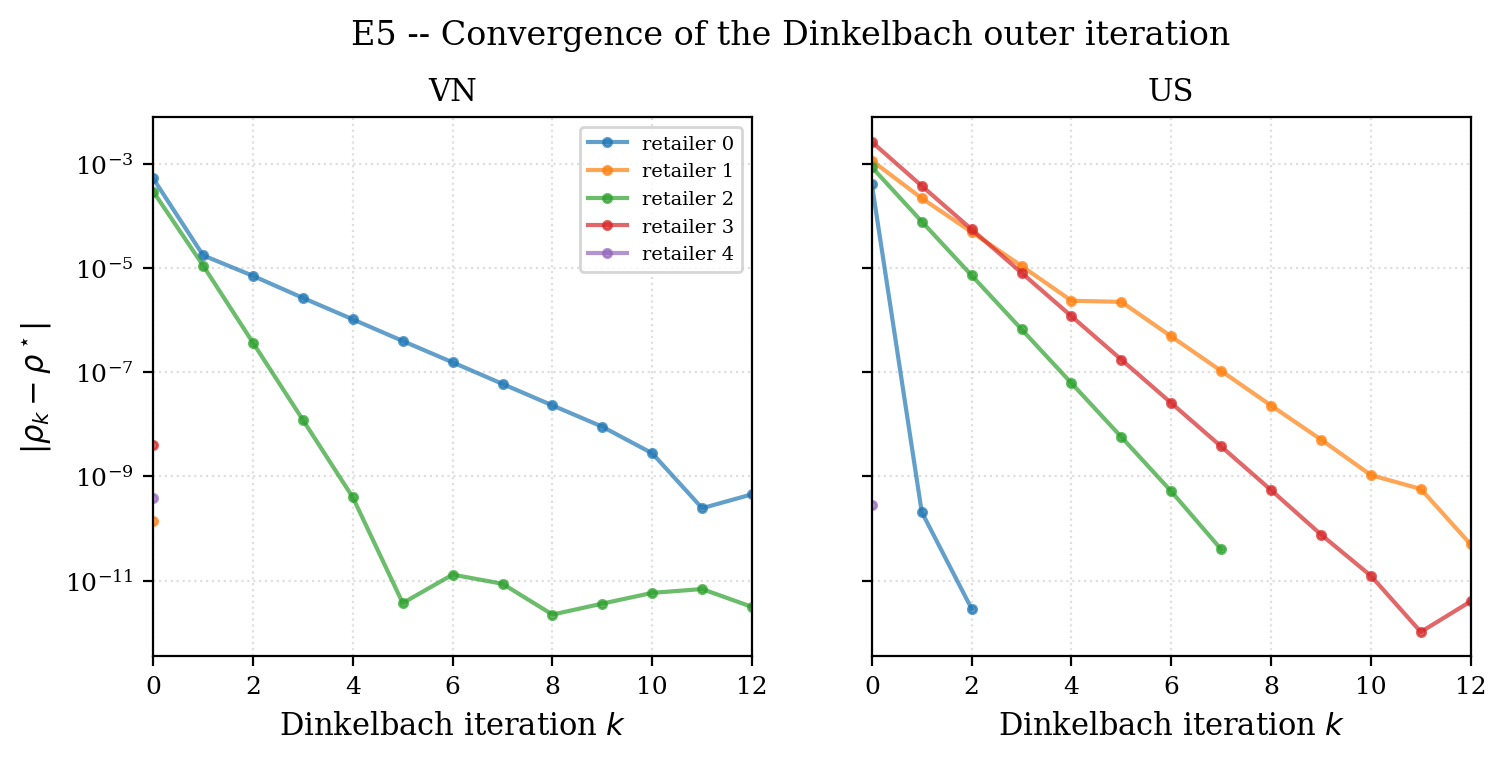}
\caption{E5A - Convergence of the Dinkelbach outer iteration
(Theorem~\ref{thm:dink-rate}). Straight lines on the semi-log axis
confirm the $R$-linear rate; retailers meeting the exactness condition
converge in two steps.}\label{fig:dink-conv}
\end{figure}

\begin{figure}[H]\centering
    \includegraphics[width=0.92\linewidth]{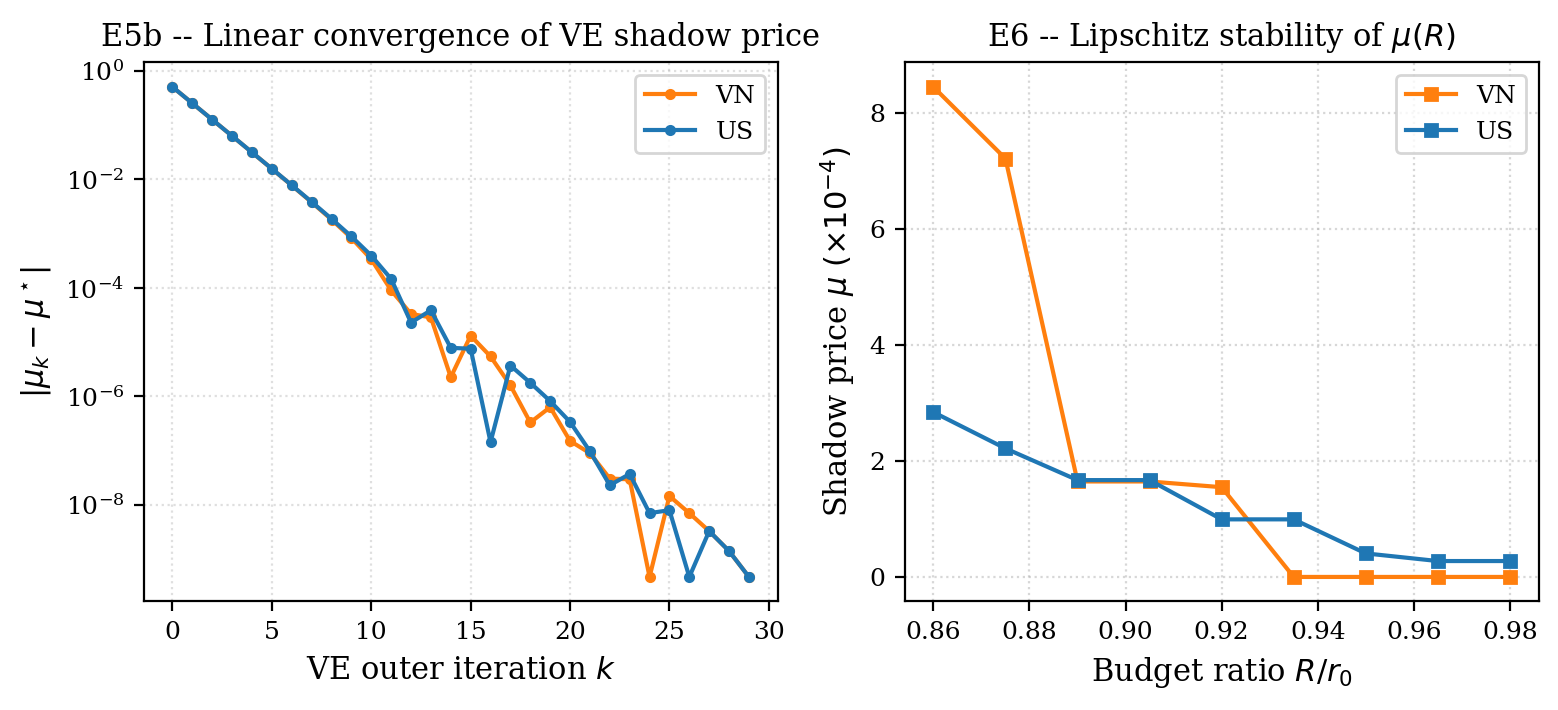}
\caption{Left (E5B): linear convergence of the VE shadow-price
iteration (Theorem~\ref{thm:ve-rate}), multiplier error halving each
step. Right (E6): the shadow price $\mu(R)$ is a smooth, monotone,
Lipschitz function of the shared risk budget $R$
(Theorem~\ref{thm:stability}); the emerging-market (VN) curve lies
above the developed-market (US) curve at every binding
budget.}\label{fig:mu}
\end{figure}

\subsection{E6 - Stability of the shadow price}\label{ssec:e6}

\paragraph{Design.} At the binding operating point we perturb the risk
budget over $R/r_0\in[0.86,0.98]$ and record the resulting VE shadow
price $\mu(R)$ for both calibrations, then estimate the Lipschitz
modulus $L_\mu=\max|\mathrm{d}\mu/\mathrm{d}R|$ by finite differences.

\paragraph{Results.} The map $\mu(R)$ is monotone decreasing and smooth
(Figure~\ref{fig:mu}, right), confirming the Lipschitz stability of
Theorem~\ref{thm:stability}: the estimated modulus is
$L_\mu^{\mathrm{VN}}=9.1{\times}10^{-5}$ and
$L_\mu^{\mathrm{US}}=6.9{\times}10^{-6}$ per unit of USD$^2$ budget.
At the reference budget $R=0.92\,r_0$ the emerging-market shadow price
$\mu_{\mathrm{VN}}=1.55{\times}10^{-4}$ exceeds its developed-market
counterpart $\mu_{\mathrm{US}}=0.99{\times}10^{-4}$, and this ordering
$\mu_{\mathrm{VN}}>\mu_{\mathrm{US}}$ persists across every binding
budget in the sweep. The bounded modulus certifies that the
shadow-price signal used for the investment recommendation of
Section~\ref{sec:conc} is robust to small mis-specification of $R$:
a $5\%$ budget error moves $\mu$ by at most a proportional amount.

\paragraph{Summary.} The six experiments jointly establish: (E1) on
every tested instance the pure-SOCP HSDI attains the same value as the
multi-start non-linear baselines to within $5.5{\times}10^{-3}$,
deterministically and without restarts, returning a certified global
optimum on the exact region; (E2) the one-factor assumption is realised
in the data with $>99\%$ spectral share; (E3) HSDI is near-optimal on
$42$-$77\%$ of the
calibration sweep and the closed-form error bound is loose only by a
constant factor elsewhere; (E4) the
empirical scaling exponent of HSDI ($\hat{\gamma}\approx 0.93$) is
significantly below that of SLSQP ($\hat{\gamma}\approx 1.30$),
consistent with the polynomial theory; (E5) both outer iterations
exhibit the predicted linear rates ($\kappa\approx0.4$ for Dinkelbach,
$0.5$ for the VE multiplier); and (E6) the shadow price $\mu(R)$ is
Lipschitz-stable with $\mu_{\mathrm{VN}}>\mu_{\mathrm{US}}$. Although
SLSQP remains faster in absolute wall time on every tested cell, so the
asymptotic advantage of HSDI is a prediction beyond the tested range,
HSDI is the only method here that pairs a feasible allocation with a
convergence guarantee and a stable shadow price.

\section{Conclusion}\label{sec:conc}

We establish a quadratic-fractional efficiency function for
cold-chain allocation, embed it in a GNEP with a shared risk-budget
constraint, and adopt the VE as the solution concept.
Theorem~\ref{thm:joint-diag} shows that the one-factor structure -
quantitatively anchored to the Arrhenius coefficient $Q_{10}$ -
enables a simultaneous diagonalisation of $(\Sigma_i,H_i)$ without
commutativity. Theorem~\ref{thm:socp-exact} reduces the Dinkelbach
inner problem to an SOCP with an explicit sufficient exactness
condition and a closed-form error bound. Section~\ref{sec:conv} then
shows that HSDI converges globally---$Q$-superlinearly on the exact
region and $R$-linearly otherwise for the Dinkelbach loop, and
$R$-linearly for the coupled VE iteration---in polynomial time, and
that its primal--dual output is Lipschitz-stable. Together these give
a deterministic, convergence-certified allocation with a shadow price
that, on the exact region, is provably globally optimal---properties
that no multi-start non-linear baseline can offer, however many
restarts it uses.

The economic interpretation is concentrated in the Lagrange multiplier
$\mu$ on the shared budget. A larger $\mu$ in less-developed markets
flags a higher shadow price per unit of additional cold-chain capacity,
identifying $\mu$ as a quantitative indicator for infrastructure
investment.

\paragraph{Managerial and policy implications.}
Building on this economic reading, the calibrated comparison delivers
actionable guidance for operators and policymakers. In the Vietnam
calibration the multiplier $\mu_{\mathrm{VN}}$ is larger than its
United States counterpart $\mu_{\mathrm{US}}$, signalling a higher
marginal value of cold-chain capacity expansion in the emerging-market
setting. This finding is consistent with broader evidence that
logistics performance gaps in Vietnam carry measurable economic costs
\citep{VuDoan2024}, and points to cold-chain capacity as a high-return
investment target in the Vietnamese retail food sector. Operators and policymakers active in markets with high $\mu$ should
prioritise capacity expansion. The SOCP dual solution converts $\mu$
into a quantitative threshold above which additional investment
becomes self-financing under the calibrated risk budget. Because HSDI
is deterministic, convergence-certified and---on the exact
region---globally optimal, and because $\mu$ is Lipschitz-stable in the
budget (Theorem~\ref{thm:stability}), the shadow price is a reliable
decision input for infrastructure planning rather than an artefact of a
randomly restarted local solver. Extending the calibration with IoT lane-level data and lifting the
model to a multi-period horizon, both discussed below as limitations,
would allow $\mu$ to be tracked and updated dynamically as the network
evolves.

\paragraph{Methodological implication.} The same pattern applies more
broadly: a physical principle yields an algebraic structure, which in
turn enables an optimisation reformulation. In our case, the Arrhenius
law produces a rank-one-plus-diagonal covariance, which the
Cholesky-spectral transform exploits without commutativity. Analogous
patterns include power-grid management with Kirchhoff structure,
heat-routing problems with Laplacian-spectrum structure, and queueing
optimisation with Toeplitz-plus-diagonal structure.

\paragraph{Limitations.} (i) $\sigma_{\mathrm{VN}}$ uses a proxy in
the absence of IoT lane-level data; (ii) the one-factor model is an
approximation; (iii) the single-period framework abstracts away
multi-period dynamics; (iv) experiments are run at $\beta=0$ to
isolate the technical kernel; (v) the scaling experiment of
Section~\ref{ssec:e4} covers $n\cdot d\in[45,180]$ with
$n_{\mathrm{inst}}=3$ LHS perturbations per cell, which is sufficient
to separate the HSDI and SLSQP scaling exponents at the $5\%$ level
(nearly disjoint $95\%$ intervals), though the predicted wall-time
cross-over has not been observed in the tested range; a larger sweep
would sharpen both the exponent estimate and the cross-over
prediction.
(vi)~The number of heat-sensitivity subgroups $K$ is fixed
exogenously at the values reported in Tables~\ref{tab:calib-vn}
and~\ref{tab:calib-us}; no sensitivity analysis over $K$ has been
conducted. Because $K$ determines the dimension $d$ of the allocation
vector and the sparsity pattern of $\Sigma_i$, the rank-one-plus-diagonal
structure exploited by HSDI may degrade as $K$ grows or as
subgroups are merged at coarser aggregation levels. A systematic
sweep over $K$ would clarify how the exactness condition of
Theorem~\ref{thm:socp-exact} tightens or weakens with the
granularity of the product taxonomy.
(vii)~The empirical evaluation rests on a small number of instances:
two market calibrations (Vietnam and the United States) and a scaling
grid with at most three Latin-hypercube perturbations per cell. While
this is adequate to verify correctness and to estimate scaling slopes,
it falls short of the sample sizes required to draw statistically
robust conclusions about generalisation across supply-chain
configurations, retailer asymmetries, or geographic markets beyond
those considered. Out-of-sample performance on held-out calibrations
remains uncharacterised.
(viii)~All parameter inputs - retail prices, spoilage coefficients,
thermal-noise standard deviations, and the shared risk budget $R$ -
are \emph{static} quantities drawn from annual reports, government
surveys, and industry benchmarks compiled at a single point in time.
The model therefore cannot capture intra-day or seasonal price
variation, demand shocks, or evolving cold-chain conditions that
practitioners monitor through real-time telemetry. In particular,
the Arrhenius coefficient $Q_{10}$ and the thermal noise $\sigma$
are treated as constants, whereas in live operations both fluctuate
continuously with ambient temperature, equipment age, and lane-level
traffic. The absence of real-time, high-frequency inputs means that
the calibrated shadow price $\mu$ reflects a long-run average rather
than an actionable real-time signal.
(ix)~The precise dimensional threshold at which the outer Dinkelbach
loop transitions from the $Q$-superlinear regime of
Theorem~\ref{thm:dink-rate}(b) to the merely $R$-linear regime of
Theorem~\ref{thm:dink-rate}(c) has not been empirically quantified.
This transition is governed by the emergence of a non-empty index set
$\mathcal{J}_{-}(\rho)=\{j:\lambda_{i,j}<\rho\}$ along the iterates,
whose cardinality is expected to grow with the decision dimension
$d=KT$ and hence with the subgroup count $K$. Because no systematic
sweep in $K$ has been conducted (see~(vi)), the rate diagnostics of
Section~\ref{ssec:e5} verify the predicted regimes only at the
calibrated values of $K$ and cannot locate the operational boundary
beyond which the superlinear guarantee ceases to hold. A dedicated
$K$-sweep, paired with per-iterate monitoring of
$|\mathcal{J}_{-}(\rho_k)|$, would map this boundary and clarify how
sensitive the observed rate is to the granularity of the product
taxonomy.
(x)~The convergence guarantees of Sections~\ref{ssec:dink-conv}
and~\ref{ssec:ve-conv} are established for a single planning period.
Extending HSDI to the rolling-horizon or IoT-driven multi-period
setting sketched in the future-research directions below would
naturally warm-start each period's outer loops from the terminal
iterate of the preceding period, and this introduces a
\emph{cumulative-convergence} risk that the present single-period
analysis does not cover. Numerical residual carried over from period
$t$ may push the initial iterate at period $t+1$ outside the local
basin of $Q$-superlinear attraction identified in
Theorem~\ref{thm:dink-rate}(b), and repeated compounding across many
periods may erode the asymptotic $R$-linear rate of
Theorem~\ref{thm:ve-rate}, producing progressive solution drift and,
in the worst case, loss of asymptotic stability of the equilibrium
trajectory. Sufficient conditions on the inter-period residual under
which global convergence to the true sequence of variational
equilibria is preserved, together with an empirical assessment of
drift on synthetic multi-period instances, lie outside the scope of
the present paper.
Each item above is a natural direction for follow-up work.

\paragraph{Future research directions.}
Several extensions follow directly from the limitations above.
First, replacing the static parameter inputs with a real-time data
pipeline - drawing on IoT sensors, electronic shelf labels, and
lane-level telematics - would transform HSDI from a planning-horizon
tool into an operational decision engine. This requires integrating
the SOCP solver with a rolling-horizon or model-predictive control
framework so that the shadow price $\mu$ is recomputed as new
temperature readings arrive. Second, a rigorous sensitivity analysis
over the subgroup count $K$ would map out the boundary of applicability
of the rank-one-plus-diagonal assumption and guide practitioners on
how finely to partition the product taxonomy before the structural
guarantee breaks down. Third, expanding the empirical evaluation to
a larger and more diverse set of markets - including cold chains
in Southeast Asian emerging economies beyond Vietnam - would
provide the out-of-sample evidence needed to assess the transferability
of the calibrated parameters and the tightness of the SOCP relaxation
across heterogeneous network topologies. Finally, lifting the model
to a multi-period or stochastic setting, where $Q_{10}$ and $\sigma$
are themselves random processes, would bring the theoretical framework
closer to the operational reality of cold-chain management under
climate variability. In parallel, a formal sensitivity analysis of the
primal-dual pair $(x^{\star},\mu)$ under joint perturbation of $R$,
$b_k$ and $H_i$ can be carried out through the second-order
variational-analysis toolbox for parameterised conic programs developed
in \citet{MordukhovichOutrataRamirez2015a,MordukhovichOutrataRamirez2015b,
HangMordukhovichSarabi2019,BenkoGfrererMordukhovich2019,
NguyenDissertation2020}. Delivering quantitative Lipschitz moduli for
the shadow price $\mu$ would rigorously certify the operational-reliability
claims made in the present paper and connect the HSDI framework to the
mainstream stability theory of nonpolyhedral conic programs.

\bibliographystyle{plainnat}

\begin{thebibliography}{50}
\sloppy
\bibitem[Akkaya and P{\i}nar(2025)]{AkkayaPinar2025} Akkaya, D. and P{\i}nar, M. {\c C}. (2025). Hidden conic quadratic representation in some generalized Nash equilibrium problems. \emph{Journal of Optimization Theory and Applications}, in submission (manuscript JOTA-D-25-01156).
\bibitem[Almgren and Chriss(2001)]{AlmgrenChriss2001} Almgren, R. and Chriss, N. (2001). Optimal execution of portfolio transactions. \emph{Journal of Risk}, 3(2):5-40.
\bibitem[Banker et al.(1984)]{BankerCharnesCooper1984} Banker, R. D., Charnes, A. and Cooper, W. W. (1984). Some models for estimating technical and scale inefficiencies in data envelopment analysis. \emph{Management Science}, 30(9):1078-1092.
\bibitem[Bao et al.(2024)]{BaoMordukhovichNam2024} Bao, N. X. D., Mordukhovich, B. S. and Nam, N. M. (2024). Nash equilibrium and minimax theorems via variational tools of convex analysis. arXiv:2408.14433 [math.OC].
\bibitem[Benko et al.(2019)]{BenkoGfrererMordukhovich2019} Benko, M., Gfrerer, H. and Mordukhovich, B. S. (2019). Characterizations of tilt-stable minimizers in second-order cone programming. \emph{SIAM Journal on Optimization}, 29(4):3100-3130.
\bibitem[Beck and Pan(2017)]{BeckPan2017} Beck, A. and Pan, D. (2017). A branch and bound algorithm for nonconvex quadratic optimization with ball and linear constraints. \emph{Journal of Global Optimization}, 69(2):309-342.
\bibitem[Ben-Tal and Nemirovski(2001)]{BenTalNemirovski2001} Ben-Tal, A. and Nemirovski, A. (2001). \emph{Lectures on Modern Convex Optimization}. SIAM.
\bibitem[Charnes, Cooper and Rhodes(1978)]{CCR1978} Charnes, A., Cooper, W. W. and Rhodes, E. (1978). Measuring the efficiency of decision making units. \emph{EJOR}, 2(6):429-444.
\bibitem[Cook and Seiford(2009)]{CookSeiford2009} Cook, W. D. and Seiford, L. M. (2009). Data envelopment analysis (DEA)-Thirty years on. \emph{European Journal of Operational Research}, 192(1):1-17.
\bibitem[Crouzeix and Ferland(1991)]{CrouzeixFerland1991} Crouzeix, J.-P. and Ferland, J. A. (1991). Algorithms for generalized fractional programming. \emph{Mathematical Programming}, 52(1-3):191-207.
\bibitem[Dembo et al.(1982)]{DemboEisenstatSteihaug1982} Dembo, R. S., Eisenstat, S. C. and Steihaug, T. (1982). Inexact Newton methods. \emph{SIAM Journal on Numerical Analysis}, 19(2):400-408.
\bibitem[Dinkelbach(1967)]{Dinkelbach1967} Dinkelbach, W. (1967). On nonlinear fractional programming. \emph{Management Science}, 13(7):492-498.
\bibitem[Endara et al.(2023)]{EndaraWiedmannAdalja2023} Endara, P., Wiedmann, M. and Adalja, A. (2023). Consumer willingness to pay for shelf life of high-temperature, short-time-pasteurized fluid milk. \emph{Journal of Dairy Science}, 106(9):5940-5957.
\bibitem[Facchinei and Kanzow(2010)]{FacchineiKanzow2010} Facchinei, F. and Kanzow, C. (2010). Generalized Nash equilibrium problems. \emph{Annals of Operations Research}, 175(1):177-211.
\bibitem[Facchinei and Pang(2003)]{FacchineiPang2003} Facchinei, F. and Pang, J.-S. (2003). \emph{Finite-Dimensional Variational Inequalities and Complementarity Problems}. Springer.
\bibitem[Fan et al.(2008)]{FanFanLv2008} Fan, J., Fan, Y. and Lv, J. (2008). High-dimensional covariance matrix estimation using a factor model. \emph{Journal of Econometrics}, 147(1):186-197.
\bibitem[FiinGroup(2025)]{FiinGroup2025} FiinGroup. \emph{Vietnam Cold Chain Market Report 2025}.
\bibitem[GCCA(2020)]{GCCA2020} Global Cold Chain Alliance. \emph{Global Cold Storage Capacity Report 2020}.
\bibitem[GCCA(2024)]{GCCA2024} Global Cold Chain Alliance. \emph{Global Top 25 List of Refrigerated Warehousing Providers 2024}.
\bibitem[Hang et al.(2019)]{HangMordukhovichSarabi2019} Hang, N. T. V., Mordukhovich, B. S. and Sarabi, M. E. (2019). Second-order variational analysis in second-order cone programming. \emph{Mathematical Programming, Series A}, 180(1-2):75-116.
\bibitem[Horn and Johnson(2012)]{HornJohnson2012} Horn, R. A. and Johnson, C. R. (2012). \emph{Matrix Analysis}, 2nd ed. CUP.
\bibitem[Iancu and Trichakis(2014)]{IancuTrichakis2014} Iancu, D. A. and Trichakis, N. (2014). Pareto efficiency in robust optimization. \emph{Management Science}, 60(1):130-147.
\bibitem[Ibaraki(1983)]{Ibaraki1983} Ibaraki, T. (1983). Parametric approaches to fractional programs. \emph{Mathematical Programming}, 26(3):345-362.
\bibitem[IMARC(2024)]{IMARC2024} IMARC Group. \emph{Vietnam Dairy Market Report 2024-2033}.
\bibitem[Korpelevich(1976)]{Korpelevich1976} Korpelevich, G. M. (1976). The extragradient method for finding saddle points and other problems. \emph{Ekonomika i Matematicheskie Metody}, 12:747-756.
\bibitem[Kulkarni and Shanbhag(2012)]{KulkarniShanbhag2012} Kulkarni, A. A. and Shanbhag, U. V. (2012). On the variational equilibrium as a refinement of the generalized Nash equilibrium. \emph{Automatica}, 48(1):45-55.
\bibitem[Lobo et al.(1998)]{Lobo1998} Lobo, M. S., Vandenberghe, L., Boyd, S. and Lebret, H. (1998). Applications of second-order cone programming. \emph{Linear Algebra Appl.}, 284(1-3):193-228.
\bibitem[Mordukhovich(2024)]{Mordukhovich2024} Mordukhovich, B. S. (2024). \emph{Second-Order Variational Analysis in Optimization, Variational Stability, and Control: Theory, Algorithms, Applications}. Springer Series in Operations Research and Financial Engineering. Springer, Cham. \url{https://doi.org/10.1007/978-3-031-53476-8}.
\bibitem[Mordukhovich et al.(2015a)]{MordukhovichOutrataRamirez2015a} Mordukhovich, B. S., Outrata, J. V. and Ram\'irez C., H. (2015a). Graphical derivatives and stability analysis for parameterized equilibria with conic constraints. \emph{Set-Valued and Variational Analysis}, 23(4):687-704.
\bibitem[Mordukhovich et al.(2015b)]{MordukhovichOutrataRamirez2015b} Mordukhovich, B. S., Outrata, J. V. and Ram\'irez C., H. (2015b). Second-order variational analysis in conic programming with applications to optimality and stability. \emph{SIAM Journal on Optimization}, 25(1):76-101.
\bibitem[Nagurney et al.(2017)]{NagurneyYuBesik2017} Nagurney, A., Yu, M. and Besik, D. (2017). Supply chain network capacity competition with outsourcing. \emph{JOGO}, 69(1):231-254.
\bibitem[Nesterov and Todd(1997)]{NesterovTodd1997} Nesterov, Y. and Todd, M. J. (1997). Self-scaled barriers and interior-point methods. \emph{Math. Oper. Res.}, 22(1):1-42.
\bibitem[Nguyen(2020)]{NguyenDissertation2020} Nguyen, H. T. V. (2020). \emph{Variational Analysis in Second-Order Cone Programming and Applications}. PhD dissertation, Wayne State University, Detroit, MI.
\bibitem[P{\'o}lik and Terlaky(2007)]{PolikTerlaky2007} P\'olik, I. and Terlaky, T. (2007). A survey of the S-lemma. \emph{SIAM Review}, 49(3):371-418.
\bibitem[Pataki(2013)]{Pataki2013} Pataki, G. (2013). Strong duality in conic linear programming. \emph{Springer Proc.\ Math.\ Stat.}, 50:613-634.
\bibitem[Ramachandran et al.(2024)]{Ramachandran2024} Ramachandran, A., Shu, K. and Wang, A. L. (2024). Hidden convexity, optimization, and algorithms on rotation matrices. \emph{Math. Oper. Res.}, forthcoming.
\bibitem[Rosen(1965)]{Rosen1965} Rosen, J. B. (1965). Existence and uniqueness of equilibrium points for concave $n$-person games. \emph{Econometrica}, 33(3):520-534.
\bibitem[Sahinidis and Grossmann(1991)]{SahinidisGrossmann1991} Sahinidis, N. V. and Grossmann, I. E. (1991). Convergence properties of generalized Benders decomposition. \emph{Comput. Chem. Eng.}, 15(7):481-491.
\bibitem[Schaible(1976)]{Schaible1976} Schaible, S. (1976). Fractional programming I, duality. \emph{Management Science}, 22(8):858-867.
\bibitem[Sharpe(1963)]{Sharpe1963} Sharpe, W. F. (1963). A simplified model for portfolio analysis. \emph{Management Science}, 9(2):277-293.
\bibitem[Tseng(1995)]{Tseng1995} Tseng, P. (1995). On linear convergence of iterative methods for the variational inequality problem. \emph{Journal of Computational and Applied Mathematics}, 60(1-2):237-252.
\bibitem[Tsiros and Heilman(2005)]{TsirosHeilman2005} Tsiros, M. and Heilman, C. M. (2005). The effect of expiration dates on purchasing behavior. \emph{Journal of Marketing}, 69(2):114-129.
\bibitem[USDA-AMS(2024)]{USDA_AMS2024} USDA Agricultural Marketing Service. \emph{Retail Milk Prices Report, November 2024}.
\bibitem[USDA-ERS(2024)]{USDA_ERS2024} USDA Economic Research Service. \emph{Dairy Market Overview 2024}.
\bibitem[USDA-FAS(2024)]{USDAFAS2024} USDA Foreign Agricultural Service. \emph{Vietnam Retail Foods Annual}, GAIN VM2024-0032.
\bibitem[van Boekel(2022)]{vanBoekel2022} van Boekel, M. A. J. S. (2022). Kinetics of heat-induced changes in dairy products. \emph{International Dairy Journal}, 126:105175.
\bibitem[Veritiv(2024)]{Veritiv2024} Veritiv Corporation. \emph{Cold Chain Best Practices Benchmark 2024}.
\bibitem[Vinamilk(2024)]{Vinamilk2024} Vinamilk JSC. \emph{Annual Report 2024}.
\bibitem[Vu and Doan(2024)]{VuDoan2024} Vu, L. H. and Doan, T. H. (2024). Effects of logistics performance on Vietnam's exports: a quantitative analysis using the PPML method. \emph{Journal of International Economics and Management}, 24(1):1-16. \url{https://doi.org/10.38203/jiem.024.1.0077}.
\bibitem[Xia, Wang and Sheu(2015)]{XiaWangSheu2015} Xia, Y., Wang, S. and Sheu, R.-L. (2015). S-lemma with equality. \emph{Math. Program.}, 156(1):513-547.
\end{thebibliography}

\appendix
\section{Detailed proofs}\label{app:proofs}

\subsection*{A.1 \quad Full proof of Theorem~\ref{thm:joint-diag}}

\emph{(i) Cholesky.} Under Assumption~\ref{ass:reg}, $H_i\in\PD^{d}$,
so there exists a unique $L_i\in\R^{d\times d}$ lower triangular with
positive diagonal such that $H_i=L_iL_i^{\top}$.

\emph{(ii) PSD of $S_i$.} Let $S_i=L_i^{-1}\Sigma_iL_i^{-\top}$.
Symmetry: $S_i^{\top}=L_i^{-1}\Sigma_i^{\top}L_i^{-\top}=S_i$. For
$z\in\R^{d}$, $z^{\top}S_iz=(L_i^{-\top}z)^{\top}\Sigma_i(L_i^{-\top}z)
\ge 0$ since $\Sigma_i\in\PSD^{d}$.

\emph{(iii) Spectral.} By the spectral theorem,
$S_i=U_i\Lambda_iU_i^{\top}$ with $U_i$ orthogonal and
$\lambda_{i,j}\ge 0$.

\emph{(iv) The map $T_i$.} Set $T_i=L_i^{-\top}U_i$, invertible with
$T_i^{-1}=U_i^{\top}L_i^{\top}$.

\emph{(v)-(vi) The two quadratic identities} follow directly:
$x_i^{\top}H_ix_i=\|L_i^{\top}x_i\|_{2}^{2}=\|y_i\|_{2}^{2}$ and
$x_i^{\top}\Sigma_ix_i=y_i^{\top}\Lambda_iy_i=\sum_j\lambda_{i,j}y_{i,j}^{2}$.
$\square$

\subsection*{A.2 \quad Full proof of Theorem~\ref{thm:socp-exact}}

\emph{(a)} If $\lambda_{i,j}\ge\rho$ for all $j$, every coefficient
$\lambda_{i,j}-\rho$ in the SOCP objective is non-negative, so the
optimisation drives each $z_{i,j}$ down to its lower bound
$y_{i,j}^{2}$. The SOCP value at optimum then equals the
Dinkelbach inner value at $x_i^{\star}=T_iy_i^{\star}$.

\emph{(b)} Define $R_{T_i}=\max_{x\in\mathcal{X}_i}\|T_i^{-1}x\|_{\infty}
<\infty$ (compactness and continuity of $T_i^{-1}$). For
$j\in\mathcal{J}_{-}(\rho)$, the coefficient $\lambda_{i,j}-\rho<0$
pushes $z_{i,j}$ upward subject to $z_{i,j}\ge y_{i,j}^{2}$ and the
implicit bound $|y_{i,j}|\le R_{T_i}$, so $z_{i,j}^{\star}\le R_{T_i}^{2}$.
The gap is bounded by
$\sum_{j\in\mathcal{J}_{-}}|\lambda_{i,j}-\rho|(z_{i,j}^{\star}-
y_{i,j}^{\star 2})\le R_{T_i}^{2}\sum_{j\in\mathcal{J}_{-}}|\lambda_{i,j}-\rho|$. $\square$

\subsection*{A.3 \quad Proof of Proposition~\ref{prop:ve-exists}}

Each $\varphi_i$ is a ratio of two quadratics with positive denominator
on $\mathcal{C}$, hence smooth. $\mathcal{C}$ is the intersection of
the polyhedron $\prod_i\mathcal{X}_i$ with the sublevel set of the
convex function $g$, hence convex and compact.
\citet[Thm.~2.3.3]{FacchineiPang2003} then gives the existence of a
VE. Strict monotonicity of $F$ implies uniqueness by the standard
contradiction with~\eqref{eq:vi}. $\square$

\subsection*{A.4 \quad Proof of Lemma~\ref{lem:psi-reg} and Theorem~\ref{thm:dink-rate}}

\emph{Lemma~\ref{lem:psi-reg}.} For fixed $x$, $\rho\mapsto
N_i(x)-\rho D_i(x)$ is affine and decreasing since $D_i(x)>0$; the
pointwise infimum $\psi_i$ is therefore concave and non-increasing, and
strictly decreasing because $D_i\ge\gamma_i>0$. Lipschitz continuity on
a bounded interval follows from $|\psi_i(\rho)-\psi_i(\rho')|\le
\bar D_i|\rho-\rho'|$ with $\bar D_i=\max_{x\in\mathcal{X}_i}D_i(x)<
\infty$. Where the minimiser $x_i(\rho)$ is unique, Danskin's theorem
gives $\psi_i'(\rho)=-D_i(x_i(\rho))\in[-\bar D_i,-\gamma_i]$;
non-uniqueness occurs only at $\rho$ with $\det M_i(\rho)=0$, a finite
set. Hence $\psi_i$ is $C^1$ and strictly decreasing near
$\rho_i^{\star}$ with a Lipschitz inverse. $\square$

\emph{Theorem~\ref{thm:dink-rate}.} (a) The Dinkelbach update solves
$\psi_i(\rho_k)+\psi_i'(\rho_k)(\rho_{k+1}-\rho_k)=0$ with
$\psi_i'(\rho_k)=-D_i(x_k)$, i.e.
$\rho_{k+1}=\rho_k-\psi_i(\rho_k)/\psi_i'(\rho_k)=N_i(x_k)/D_i(x_k)$,
which is Newton's method for the concave decreasing $\psi_i$; Newton on
a concave function started above the root produces a monotone
decreasing sequence converging to $\rho_i^{\star}$
\citep{CrouzeixFerland1991,Ibaraki1983}. (b) Since $\psi_i$ is $C^1$
with $\psi_i'(\rho_i^{\star})\le-\gamma_i<0$, Newton converges
$Q$-superlinearly; if $\psi_i'$ is locally Lipschitz with modulus
$\ell$, the standard Newton estimate gives
$\rho_{k+1}-\rho_i^{\star}\le(\ell/2\gamma_i)(\rho_k-\rho_i^{\star})^2$,
so $C=\ell/2\gamma_i$. (c) The Dinkelbach step with the relaxation solves
$\psi_i^{R}(\rho_k)+\psi_i'(\rho_k)(\rho_{k+1}-\rho_k)=0$, which is an
inexact Newton step for $\psi_i(\rho)=0$ with residual
$\psi_i(\rho_k)-\psi_i^{R}(\rho_k)\ge 0$. Under the relative forcing
bound $\psi_i(\rho_k)-\psi_i^{R}(\rho_k)\le\kappa_i|\psi_i(\rho_k)|$
with $\kappa_i<1$, the inexact-Newton convergence theorem of
\citet{DemboEisenstatSteihaug1982}, applied to the $C^1$ strictly
monotone $\psi_i$ of Lemma~\ref{lem:psi-reg}, gives local convergence
to $\rho_i^{\star}$ with asymptotic linear factor at most $\kappa_i$;
on the exact region $\kappa_i=0$ and part~(b) is recovered. $\square$

\subsection*{A.5 \quad Proof of Proposition~\ref{prop:r-mono} and Theorem~\ref{thm:ve-rate}}

\emph{Proposition~\ref{prop:r-mono}.} For $\mu'>\mu$, optimality of
$x_i(\mu),x_i(\mu')$ for the penalised objectives
$\varphi_i+\mu\,q_i$ and $\varphi_i+\mu'q_i$ with
$q_i(x)=x^{\top}\Sigma_ix$ gives, on adding the two optimality
inequalities, $(\mu'-\mu)\big(q_i(x_i(\mu'))-q_i(x_i(\mu))\big)\le 0$,
so $q_i(x_i(\cdot))$ is non-increasing; summing over $i$ shows $r$ is
non-increasing, strictly so where the responses are interior.
Existence and uniqueness of the binding root follow from continuity and
strict monotonicity. $\square$

\emph{Theorem~\ref{thm:ve-rate}.} (a) With $F$ $m_F$-strongly monotone
and $L_F$-Lipschitz, for $z,z'\in\mathcal{C}$ the projected residual
map $T_\tau(z)=P_{\mathcal{C}}(z-\tau F(z))$ obeys
$\|T_\tau(z)-T_\tau(z')\|^2\le(1-2\tau m_F+\tau^2L_F^2)\|z-z'\|^2$ by
non-expansiveness of $P_{\mathcal{C}}$ and strong monotonicity; the
bracket is $<1$ for $\tau\in(0,2m_F/L_F^2)$ and minimised at
$\tau=m_F/L_F^2$, giving $\eta=\sqrt{1-m_F^2/L_F^2}$
\citep{Tseng1995,FacchineiPang2003}. (b) By
Proposition~\ref{prop:r-mono}, $r$ is continuous and strictly
decreasing on the binding branch, so bisection halves the multiplier
bracket each step, $|\mu_k-\mu^{\star}|\le 2^{-k}|\mu_0-\mu^{\star}|$;
where $r$ is differentiable with $r'(\mu^{\star})<0$, the safeguarded
Newton step is locally $Q$-superlinear. $\square$

\subsection*{A.6 \quad Proof of Theorem~\ref{thm:stability}}

The KKT system of~\eqref{eq:socp-relax} together with the scalar
complementarity~\eqref{eq:mu-system} is a parameterised generalised
equation $0\in G(w;p)+N_{\mathcal{K}}(w)$ with $w=(x,\mu)$,
$p=(R,\{b_i\},\{H_i\})$ and $\mathcal{K}$ a product of the nonnegative
orthant and Lorentz cones. Under the metric-subregularity constraint
qualification, the graphical-derivative / coderivative criterion for
isolated calmness of solution maps to conic generalised equations
\citep{MordukhovichOutrataRamirez2015a,MordukhovichOutrataRamirez2015b}
holds; its specialisation to second-order cone systems
\citep{HangMordukhovichSarabi2019,NguyenDissertation2020} and the
strong-regularity calculus of \citet[Ch.~9]{Mordukhovich2024} yield
single-valued Lipschitz dependence of $w=(x^{\star},\mu)$ on $p$ in a
neighbourhood of the reference equilibrium. Restricting the
perturbation to $R$ gives $|\mu(R')-\mu(R)|\le L_\mu|R'-R|$. $\square$

\section{Calibration tables}\label{app:calib}

\begin{table}[H]\centering\small
\caption{Vietnam calibration parameters.}\label{tab:calib-vn}
\resizebox{\linewidth}{!}{%
\begin{tabular}{l l l l}
\toprule
Parameter & Value & Unit & Source \\
\midrule
Retailers $n$ & 5 (WC, BHX, Co.opmart, GO!, LotteMart) & - & \citet{USDAFAS2024} \\
Market size & 5.71 (2024) & USD bn & \citet{IMARC2024} \\
HTST retail price & 1.00-1.67 & USD/L & \citet{USDAFAS2024} \\
UHT retail price & 1.50-2.00 & USD/L & \citet{USDAFAS2024,IMARC2024} \\
Raw milk farm-gate & 0.76 & USD/L & \citet{Vinamilk2024} \\
HTST shelf life & 7-10 & days & \citet{EndaraWiedmannAdalja2023} \\
$\beta_{\mathrm{HTST}}$ & 0.04 & USD/L/day & \citet{EndaraWiedmannAdalja2023} \\
$\beta_{\mathrm{ESL}}$  & 0.02 & USD/L/day & \citet{EndaraWiedmannAdalja2023} \\
$\beta_{\mathrm{UHT}}$  & 0.005 & USD/L/day & \citet{EndaraWiedmannAdalja2023} \\
$\sigma_{\mathrm{VN}}$ (proxy) & 1.0 ([0.5,2.0] sweep) & $^{\circ}$C & \citet{FiinGroup2025} \\
Cold-storage capacity & 4.2 (2024) $\to$ 6.1 (2027) & million pallets & \citet{FiinGroup2025} \\
Risk budget $R$ & $1.2\times 10^{6}$ & USD$^{2}$ & calibrated \\
\bottomrule
\end{tabular}}
\end{table}

\begin{table}[H]\centering\small
\caption{United States calibration parameters.}\label{tab:calib-us}
\resizebox{\linewidth}{!}{%
\begin{tabular}{l l l l}
\toprule
Parameter & Value & Unit & Source \\
\midrule
Retailers $n$ & 5 (Walmart, Kroger, Albertsons, Publix, Costco) & - & \citet{USDA_ERS2024} \\
Market size & $\sim 70$ (2024) & USD bn & \citet{USDA_ERS2024} \\
HTST retail price & 1.00-1.30 & USD/L & \citet{USDA_AMS2024} \\
Raw milk farm-gate & 0.45 & USD/L & \citet{USDA_ERS2024} \\
HTST shelf life & 14-21 & days & \citet{EndaraWiedmannAdalja2023} \\
$\beta_{\mathrm{HTST}}$ & 0.06 & USD/L/day & \citet{EndaraWiedmannAdalja2023} \\
$\beta_{\mathrm{ESL}}$  & 0.03 & USD/L/day & \citet{EndaraWiedmannAdalja2023} \\
$\beta_{\mathrm{UHT}}$  & 0.01 & USD/L/day & \citet{EndaraWiedmannAdalja2023} \\
$\sigma_{\mathrm{US}}$ & 0.3 ([0.2,0.5] sweep) & $^{\circ}$C & \citet{Veritiv2024} \\
Cold-storage capacity & 5.5 (156 million m$^{3}$) & billion ft$^{3}$ & \citet{GCCA2020} \\
Risk budget $R$ & $4.5\times 10^{5}$ & USD$^{2}$ & calibrated \\
\bottomrule
\end{tabular}}
\par\smallskip\begin{minipage}{0.95\linewidth}\footnotesize\textit{Note:}
The Arrhenius coefficients are regime-independent across both
calibrations: $Q_{10,\mathrm{HTST}}=3.0$, $Q_{10,\mathrm{ESL}}=2.5$,
$Q_{10,\mathrm{UHT}}=2.0$ \citep{vanBoekel2022}.
\end{minipage}
\end{table}

\end{document}